\documentclass[11pt]{amsart}

\usepackage{amsfonts}
\usepackage{amssymb}
\usepackage{amsthm}
\usepackage{amsmath}
\usepackage{hyperref}
\usepackage{tikz}
\usepackage{pgf}
\usepackage{graphicx}
\usepackage[top=3.5cm, bottom=3.5cm, left=2.5cm, right=2.5cm]{geometry}

\title[Algebraic properties and harmonic functions]{Characterisations of algebraic properties of groups in terms of harmonic functions}
\date{}
\author{Matthew C. H. Tointon}
\thanks{The majority of this work was supported by a Junior Research Fellowship from Homerton College, University of Cambridge. This work was also partially supported by a grant from the Cambridge Philosophical Society.}

\address{Homerton College, University of Cambridge, Hills Road, Cambridge, CB2 8PH, United Kingdom}
\email{m.tointon@maths.cam.ac.uk}

\newtheorem{prop}{Proposition}[section]
\newtheorem{theorem}[prop]{Theorem}
\newtheorem{lemma}[prop]{Lemma}
\newtheorem{corollary}[prop]{Corollary}
\newtheorem{conjecture}[prop]{Conjecture}

\theoremstyle{definition}
\newtheorem{definition}[prop]{Definition}
\newtheorem{question}[prop]{Question}

\theoremstyle{remark}

\newtheorem{remark}[prop]{Remark}
\newtheorem{remarks}[prop]{Remarks}

\theoremstyle{theorem}

\newcommand{\R}{\mathbb{R}}

\newcommand{\N}{\mathbb{N}}
\newcommand{\Z}{\mathbb{Z}}

\newcommand{\E}{\mathbb{E}}
\newcommand{\K}{\mathbb{K}}

\newcommand{\Prob}{\mathbb{P}}

\newcommand{\Aut}{\text{\textup{Aut}}\,}
\newcommand{\supp}{\text{\textup{supp}}\,}

\newcommand{\mdim}{\text{\textup{mdim }}}

\numberwithin{equation}{section}

\begin{document}
\maketitle


\begin{abstract}
We prove various results connecting structural or algebraic properties of graphs and groups to conditions on their spaces of harmonic functions. In particular: we show that a group with a finitely supported symmetric measure has a finite-dimensional space of harmonic functions if and only if it is virtually cyclic; we present a new proof of a result of V. Trofimov that an infinite vertex-transitive graph admits a non-constant harmonic function; we give a new proof of a result of T. Ceccherini-Silberstein, M. Coornaert and J. Dodziuk that the Laplacian on an infinite, connected, locally finite graph is surjective; and we show that the positive harmonic functions on a non-virtually nilpotent linear group span an infinite-dimensional space.
\end{abstract}
\setcounter{tocdepth}{1}
\tableofcontents
\section{Introduction}
One can often obtain algebraic information about a group by considering it as a geometric object. For example, if $G$ is a group and $S\subset G$ is a finite, symmetric set then one can construct the \emph{Cayley graph} $(G,S)$ of $G$ with respect to $S$ by declaring the elements of $G$ to be vertices and saying that $x$ and $y$ are joined by an edge if and only if there is some non-identity element $s\in S$ such that $xs=y$.

One way of studying the geometry of a Cayley graph, or indeed any graph, is to consider the behaviour of probabilistic processes on it. In this paper we are particularly concerned with linking the algebra and geometry of groups and graphs to spaces of \emph{harmonic functions} on them.

Before we define these, let us establish some notation. A \emph{weighted graph} $\Gamma$ is a graph in which to each edge $xy$ we associate a real number $\omega_{xy}=\omega_{yx}>0$ called a \emph{weight}; the degree of a vertex $x$ is then given by $\deg x=\sum_{y\sim x}\omega_{xy}$. We define the \emph{Laplacian} $\Delta=\Delta_\Gamma$ on $\Gamma$ by setting $\Delta f(x)=f(x)-\frac{1}{\deg x}\sum_{y\sim x}\omega_{xy}f(y)$ for every function $f:\Gamma\to\R$.

If $G$ is a group then a probability measure $\mu$ on $G$ is said to be a \emph{generating probability measure} if the semigroup generated by its support, $\supp\mu$, is $G$; it is said to be \emph{symmetric} if $\mu(g)=\mu(g^{-1})$ for every $g\in G$. For a group $G$ with a finitely supported generating probability measure $\mu$ we write $\Delta=\Delta_\mu$ for the \emph{Laplacian} on $G$ with respect to $\mu$, defined by setting $\Delta f(x)=f(x)-\sum_{s\in\supp\mu}\mu(s)f(xs)$ for every function $f:G\to\R$. Note that if $\mu$ is symmetric and $\Gamma$ is the Cayley graph of $G$ with respect to $\supp\mu$, weighted such that $\omega_{xy}=\mu(x^{-1}y)$, then $\Delta_\mu=\Delta_\Gamma$. We denote this weighted Cayley graph by $(G,\mu)$.

A \emph{harmonic} function on a weighted graph $\Gamma$ or group $G$ with generating probability measure $\mu$ is defined to be a function belonging to the kernel of the corresponding Laplacian. We write $H(G,\mu)$ for the space of harmonic functions on $G$ with respect to $\mu$.

Perhaps the most famous example of a result linking the algebraic structure of a group to the geometry of a Cayley graph is M. Gromov's celebrated theorem on groups of polynomial growth, which states that a certain \emph{geometric} condition on a Cayley graph $(G,S)$ (polynomial volume growth) is characteristic of a certain \emph{algebraic} condition on the subgroup of $G$ generated by $S$ (virtual nilpotency) \cite{gromov}. A recent proof of Gromov's theorem due to B. Kleiner \cite{kleiner} provides an example of how harmonic functions are related to the algebra and geometry of groups, since a key step in Kleiner's proof is to show that if a group has polynomial growth then the vector space of harmonic functions on $(G,S)$ that grow at most linearly in the Cayley-graph distance from the identity is finite dimensional.

While Kleiner's proof of Gromov's theorem essentially uses the space of linearly growing harmonic functions as a tool to characterise an algebraic condition on a group in terms of a geometric condition, in principle it should be possible to characterise certain algebraic or geometric conditions purely in terms of spaces of harmonic functions. Indeed, in a very recent preprint, T. Meyerovitch and A. Yadin \cite{mey-yad} have shown that in the case of a finitely generated group that is linear or virtually soluble, being virtually nilpotent is \emph{equivalent} to having a finite-dimensional space of linearly growing harmonic functions. This equivalence is, moreover, conjectured to hold for all finitely generated groups \cite{mey-yad}.

The first result of the present paper shows that finite-dimensionality of the space of \emph{all} harmonic functions on a group is also equivalent to a simple algebraic condition.
\begin{theorem}\label{con:harm.fin.dim}Let $G$ be an infinite group, and let $\mu$ be a symmetric, finitely supported generating probability measure on $G$. Then the space of harmonic functions on $(G,\mu)$ is finite dimensional if and only if $G$ contains a finite-index subgroup isomorphic to $(\Z,+)$.
\end{theorem}
One can also consider, as Kleiner did in his proof of Gromov's theorem, subspaces of $H(G,\mu)$ consisting of functions of polynomial growth. Given a group $G$ with a finitely supported generating probability measure $\mu$, denote by $|g|=|g|_\mu$ the word distance of $g$ from the identity with respect to the generating set $\supp\mu$. The space $H^k(G,\mu)$ of harmonic functions on $G$ of polynomial growth of degree at most $k$ is then defined by $H^k(G,\mu)=\{h\in H(G,\mu):|h(x)|\ll_h|x|^k\text{ as }x\to\infty\}$. The set $\bigcup_{k=1}^\infty H^k(G,\mu)$ of all harmonic functions of polynomial growth on $G$ is of course also a subspace of $H(G,\mu)$. Emmanuel Breuillard has pointed out a result of Osin \cite{osin} that combines with the proof of Theorem \ref{con:harm.fin.dim} and the result of Meyerovitch and Yadin \cite{mey-yad} to give the following stronger statement, valid for all groups except perhaps those that are amenable but not elementary amenable.
\begin{corollary}\label{cor:osin}
Let $G$ be an infinite group that is either elementary amenable or non-amenable, and let $\mu$ be a symmetric, finitely supported generating probability measure on $G$. Then the space of all harmonic functions of polynomial growth on $(G,\mu)$ is finite dimensional if and only if $G$ contains a finite-index subgroup isomorphic to $(\Z,+)$.
\end{corollary}
\begin{conjecture}\label{con:osin}
Corollary \ref{cor:osin} holds for all finitely generated groups.
\end{conjecture}

A fairly immediate consequence of Theorem \ref{con:harm.fin.dim} (at least in the presence of other, standard, results) is that the space of harmonic functions on a group with a symmetric, finitely supported generating probability measure is $1$-dimensional if and only if the group is finite. In fact, this was already well known, and there exist far simpler proofs than via Theorem \ref{con:harm.fin.dim}. Indeed, V. Trofimov \cite{trofimov} shows that this characterisation holds, more generally, for vertex-transitive graphs. The second result of this paper is a new proof of Trofimov's result, valid in the even more general setting of vertex-transitive weighted graphs (Trofimov's proof could conceivably also work in this more general setting).
\begin{prop}\label{prop:trofimov}
Let $\Gamma$ be an infinite, locally finite, vertex-transitive weighted graph. Then $\Gamma$ admits a non-constant harmonic function.
\end{prop}
In Section \ref{sec:exist.of.harm.fn} we prove Proposition \ref{prop:trofimov} in the case that the random walk on $\Gamma$ is transient; for the recurrent case we refer the reader to \cite{trofimov} (where the two cases are also treated separately). See Section \ref{sec:background} for definitions of \emph{transient}, \emph{recurrent} and \emph{random walk}.
\begin{remarks}
Trofimov's result is in fact stronger than Proposition \ref{prop:trofimov}, as it proves the existence of a function whose growth rate is bounded in terms of the rate of volume growth of metric balls in $\Gamma$. Nonetheless, it seems to be of interest to have an alternative proof of the qualitative statement, and in any case we deduce Proposition \ref{prop:trofimov} fairly immediately from a slightly more general result (Proposition \ref{prop:harm.exists}, below) that is an important ingredient in our proof of Theorem \ref{con:harm.fin.dim}.

Proposition \ref{prop:trofimov} does not necessarily hold if $\Gamma$ is not vertex transitive, as can be seen by considering the graph in Figure \ref{fig:no.harm}. This example was presented explicitly in a talk of Coornaert (available at \url{http://www-irma.u-strasbg.fr/~coornaer/florence-laplacian-2012.pdf}), having been observed by Trofimov \cite[Remark 2]{trofimov}.
\begin{figure}
\caption{An infinite regular graph with no non-constant harmonic functions \cite[Remark 2]{trofimov}}\label{fig:no.harm}
\begin{center}
\begin{tikzpicture}
\draw (0,0) -- (1,0);
\draw (0,0) -- (1,1);
\draw (0,0) -- (1,-1);
\draw (1,1) -- (2,0);
\draw (1,-1) -- (2,0);
\draw (1,1) -- (1,-1);
\draw (2,0) -- (3,0);
\draw (3,0) -- (4,1);
\draw (3,0) -- (4,-1);
\draw (4,1) -- (5,0);
\draw (4,-1) -- (5,0);
\draw (4,1) -- (4,-1);
\draw (5,0) -- (6,0);
\draw (6,0) -- (7,1);
\draw (6,0) -- (7,-1);
\draw (7,1) -- (8,0);
\draw (7,-1) -- (8,0);
\draw (7,1) -- (7,-1);
\draw[dashed] (8,0) -- (9,0);
\filldraw (0,0) circle (2pt);
\filldraw (1,0) circle (2pt);
\filldraw (1,1) circle (2pt);
\filldraw (1,-1) circle (2pt);
\filldraw (2,0) circle (2pt);
\filldraw (3,0) circle (2pt);
\filldraw (4,1) circle (2pt);
\filldraw (4,-1) circle (2pt);
\filldraw (5,0) circle (2pt);
\filldraw (6,0) circle (2pt);
\filldraw (7,1) circle (2pt);
\filldraw (7,-1) circle (2pt);
\filldraw (8,0) circle (2pt);
\end{tikzpicture}
\end{center}
\end{figure}
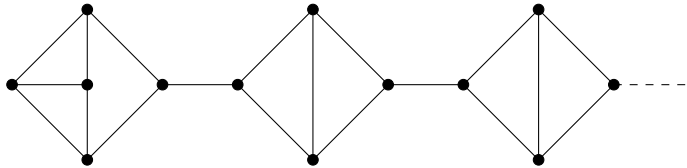
\end{remarks}
\bigskip

\noindent The results we have stated so far characterise algebraic conditions on a group in terms of the kernel of the Laplacian. Our next result characterises a structural condition on a graph in terms of the \emph{image} of the Laplacian.

A simple rank-nullity argument shows that the Laplacian on a finite graph is not surjective, since its kernel contains the constant functions. T. Ceccherini-Silberstein, M. Coornaert and J. Dodziuk \cite[Theorem 1.1]{Lap.surj.gen} show that the converse is also true for connected graphs. In Section \ref{sec:Lap.surj} we give a new proof of this result, valid in weighted graphs (Ceccherini-Silberstein, Coornaert and Dodziuk's proof could conceivably also work in this more general setting).
\begin{prop}\label{prop:CSC.Lap.surj}
The Laplacian on an infinite, connected, locally finite weighted graph $\Gamma$ is surjective onto $\R^\Gamma$. Thus, the Laplacian on a locally finite weighted graph $\Gamma$ is surjective onto $\R^\Gamma$ if and only if every connected component of $\Gamma$ is infinite.
\end{prop}

Our proof of Proposition \ref{prop:CSC.Lap.surj} is inspired by an earlier, less general, result of Ceccherini-Silberstein and Coornaert, which states that the Laplacian on an infinite Cayley graph is surjective \cite[Theorem 1.1]{Lap.surj}. We also show that this holds for more general Laplacians on groups.
\begin{prop}\label{prop:grp.Lap.surj}
Let $G$ be an infinite group, and let $\mu$ be a (not necessarily symmetric) finitely supported generating probability measure on $G$. Then $\Delta_\mu$ is surjective.
\end{prop}
It is quite likely that the argument of Ceccherini-Silberstein and Coornaert \cite{Lap.surj} could also give a proof of Proposition \ref{prop:grp.Lap.surj}. However, our proof of Proposition \ref{prop:grp.Lap.surj} is simpler than the argument of \cite{Lap.surj} (see Remarks \ref{rem:Lap.surj.cases}, below), and since Proposition \ref{prop:grp.Lap.surj} cannot be concluded directly from either Proposition \ref{prop:CSC.Lap.surj} or \cite[Theorem 1.1]{Lap.surj}, it seems, in any case, worth recording a proof here.
\bigskip

\noindent An important tool in this paper is the so-called \emph{Garden of Eden theorem} for linear cellular automata, originally due to Ceccherini-Silberstein and Coornaert \cite{GofE}. Given its importance to our arguments, we introduce it briefly here.

If $G$ is a group and $A$ is a set, called the \emph{alphabet}, then $G$ acts on the set $A^G$ of maps $f:G\to A$ via $g\cdot f(x)=f(g^{-1}x)$. If $f:G\to A$ and $M\subset G$ then we denote by $f|_M$ the restriction of $f$ to $M$. A \emph{cellular automaton} over $G$ on the alphabet $A$ is a map $\tau:A^G\to A^G$ with the property that there is some finite set $M\subset G$ and a map $\lambda:A^M\to A$ such that $\tau(f)(g)=\lambda((g\cdot f)|_M)$. The set $M$ is called a \emph{memory set} for $\tau$, and $\lambda$ is called a \emph{local defining map}.

Given an initial state $f_0\in A^G$, one can consider $\tau$ as defining a dynamical process on $A^G$ by setting $f_{i+1}=\tau(f_i)$ to obtain a sequence $f_0,f_1,f_2,\ldots$ of configurations in $A^G$. A configuration $f\in A^G$ is then said to be a \emph{Garden of Eden} configuration if it is not in the image of $\tau$, and hence can appear only as an initial configuration in this dynamical process.

The term \emph{Garden of Eden theorem} for a class of cellular automata is often used to describe a result giving a necessary and sufficient condition for the existence of Garden of Eden configurations, or, to put it another way, a necessary and sufficient condition for a cellular automaton in the class to be surjective. There are various results depending on the alphabet and the group; we refer the reader to \cite{GofE,LCA} for more detailed background to this area.

The class of interest to us is the class of \emph{linear cellular automata}, in which $A$ is a finite-dimensional vector space $V=\K^r$ over a field $\K$ and a linear cellular automaton is a cellular automaton that is also a linear map $V^G\to V^G$.
\begin{theorem}[Garden of Eden theorem for linear cellular automata; Ceccherini-Silberstein--Coornaert {\cite[Theorem 8.9.6]{LCA}}]\label{thm:GofE}
Let $V$ be a finite-dimensional vector space and let $G$ be an amenable group. Then a linear cellular automaton $\tau:V^G\to V^G$ is surjective if and only if it is pre-injective.
\end{theorem}
Here, and throughout this paper, a linear map on $V^G$ is said to be \emph{pre-injective} if its restriction to the subspace $V^G_0$ of finitely supported functions in $V^G$ is injective.

The Laplacian on a group with a finitely supported generating probability measure is an example of a linear cellular automaton, and so Theorem \ref{thm:GofE} can readily be applied to such a Laplacian, provided that the group is amenable. However, in this paper we are concerned with Laplacians on arbitrary groups, and even graphs, and so we seek a version of Theorem \ref{thm:GofE} that holds in this greater generality.

Given a locally finite graph $\Gamma$ and an alphabet $A$, we say that a map $\tau:A^\Gamma\to A^\Gamma$ is \emph{locally specifiable} if $\tau(f)(x)$ depends only on $f(x)$ and $f(y)$ for $y\sim x$. Note, in particular, that if $\tau$ is a cellular automaton on a group $G$ with memory set $M$ then $\tau$ is a locally specifiable map on the Cayley graph $(G,M\cup M^{-1})$. The Laplacian on a locally finite graph is also locally specifiable.

The result underpinning much of this paper is the following.
\begin{theorem}\label{thm:GofE.neighbourly}
Let $V$ be a finite-dimensional vector space and let $\Gamma$ be a locally finite graph. Then a locally specifiable linear map $\tau:V^\Gamma\to V^\Gamma$ is surjective if and only if its transpose $\tau'$ is pre-injective.
\end{theorem}
Here the \emph{transpose} of $\tau$ is defined in terms of the natural (possibly infinite) matrix representation of $\tau$, which we define precisely in Section \ref{sec:background}. The transpose of $\tau$ is then simply the locally specifiable linear map whose corresponding matrix is the transpose of the matrix corresponding to $\tau$.

Let us emphasise here that Theorem \ref{thm:GofE.neighbourly} applies, in particular, to linear cellular automata over non-amenable groups, and is, in that sense, considerably more general than Theorem \ref{thm:GofE}.
\begin{corollary}[Garden of Eden theorem for symmetric linear cellular automata over non-amenable groups]\label{cor:self-adj.lca}
Let $V$ be a finite-dimensional vector space and let $G$ be a (not necessarily amenable) group. Then a symmetric linear cellular automaton $\tau:V^G\to V^G$ is surjective if and only if it is pre-injective.
\end{corollary}
\begin{remarks}\label{rem:GofE}
The reader may refer to \cite[\S5]{GofE} or \cite[\S8.10-8.11]{LCA} for examples of (asymmetric) linear cellular automata on non-amenable groups for which Theorem \ref{thm:GofE} fails; thus, generalisations to non-amenable groups in the spirit of Corollary \ref{cor:self-adj.lca} must necessarily have some additional hypothesis on the map $\tau$.

With a bit more work, one can adapt some of the techniques from \cite{GofE} to recover Theorem \ref{thm:GofE} in full from Theorem \ref{thm:GofE.neighbourly}; see Appendix \ref{ap:GofE}.
\end{remarks}
\subsection*{Outline of the paper}
In Section \ref{sec:background} we give detailed definitions and present some necessary background material. In Section \ref{sec:GofE} we prove Theorem \ref{thm:GofE.neighbourly}, before applying it in Section \ref{sec:Lap.surj} to prove Propositions \ref{prop:CSC.Lap.surj} and \ref{prop:grp.Lap.surj}. In Section \ref{sec:duality.result} we use all of these results to develop a tool for proving the existence of harmonic functions on a graph or group, and then in Section \ref{sec:exist.of.harm.fn} we use this tool to prove a slight technical generalisation of Proposition \ref{prop:trofimov}. In Section \ref{sec:v.ab} we prove the easier `direct' direction of Theorem \ref{con:harm.fin.dim}, and in Section \ref{sec:linear} we prove the `inverse' direction in the case of a non-virtually nilpotent linear group. In Sections \ref{sec:v.cyc}--\ref{sec:proof} we reduce Theorem \ref{con:harm.fin.dim} to the linear case and complete the proof, as well as proving Corollary \ref{cor:osin}.

In the appendix we present two additional applications of Theorem \ref{thm:GofE.neighbourly}. In the first, we recover Theorem \ref{thm:GofE}; in the second, we reformulate a conjecture of I. Kaplansky, the so-called `stable-finiteness' conjecture.
\subsection*{Acknowledgements}
The author is grateful to Emmanuel Breuillard, Sara Brofferio, Michel Coornaert, Ben Green and Gady Kozma for helpful conversations. Thanks are also due to one anonymous referee for noticing an error in an earlier claimed proof of Theorem \ref{con:harm.fin.dim}, and another anonymous referee for helpful comments on an earlier version of this paper. Some of this work was carried out at the Institut Henri Poincar\'e, Paris, during its excellent trimester `Random walks and asymptotic geometry of groups', 2014.
\section{Background and notation}\label{sec:background}
In this section we set much of our notation and present general background material from the literature. Much of this is standard; see, for example, \cite{pete}.

Throughout this paper, by a \emph{graph} $\Gamma$ we mean an undirected weighted graph with no loops and no multiple edges. We denote by $e$ some distinguished vertex; this vertex is always the identity in the case that $\Gamma$ is a weighted Cayley graph. We write $x\sim y$ to indicate that $x$ and $y$ are neighbours. An isomorphism of weighted graphs is an isomorphism of graphs that preserves weights.  A weighted graph is called \emph{regular} if $\deg x=\sum_{y\sim x}\omega_{xy}$ is independent of the vertex $x$.

Denote by $d=d_\Gamma$ the graph metric on a graph $\Gamma$; thus, for vertices $x\ne y\in\Gamma$ the quantity $d(x,y)$ is equal to length of a path of minimum length joining $x$ to $y$. If $G$ is a group with a finitely supported generating probability measure $\mu$ then we denote by $d=d_\mu$ the graph distance on $(G,\mu)$ (thus $d_\mu$ is the word metric with respect to the generating set $\supp\mu$).

If $\Gamma$ is a graph or group and $V=\K^n$ is a vector space then for each vertex or element $x\in\Gamma$ and each $i=1,\ldots,n$ we denote by $\delta_x^i:\Gamma\to V$ the map defined by $\delta_x^i(x)=e_i$, and $\delta_x^i(y)=0$ for every $y\ne x$. In the event that $n=1$ we drop the superscript and define $\delta_x:\Gamma\to\K$ by $\delta_x(x)=1$ and $\delta_x(y)=0$ for every $y\ne x$. The $\delta_x^i$ form a basis for the space $V^\Gamma_0$ of finitely supported $V$-valued functions on $\Gamma$ (and, for the purposes of this paper, should `morally' be thought of as a basis for $V^\Gamma$).

This space $V^\Gamma_0$ is invariant under any locally specifiable linear map $\tau:V^\Gamma\to V^\Gamma$, and so we may consider the (possibly infinite) matrix of the restriction $\tau|_{V^\Gamma_0}$ with respect to this basis. In fact, $\tau$ is entirely determined by its restriction to $V^\Gamma_0$, and so the matrix of $\tau|_{V^\Gamma_0}$ with respect to this basis completely determines $\tau$. Moreover, the composition of such matrices respects the composition of the corresponding linear maps. Throughout this paper, when we refer to the matrix of a locally specifiable linear map $\tau:V^\Gamma\to V^\Gamma$ we mean the matrix of its restriction $\tau|_{V^\Gamma_0}$ with respect to the basis $\{\delta_x^i\}$.

Given a finite set $Y$ and a function $f:Y\to\R$, we generally denote by $\E_{y\in Y}$ the average $\E_{y\in Y}f(y)=\frac{1}{|Y|}\sum_{y\in Y}f(y)$. However, in the specific case that $\mu$ is a generating probability measure on a group $G$, and $S$ is the support of $\mu$, given a function $f:S\to\R$ the notation $\E_{s\in S}$ means the average $\E_{s\in S}f(s)=\sum_{s\in S}\mu(s)f(s)$. Note that these definitions agree only if $\mu$ is the uniform probability measure on $S$.

If $G_1,G_2$ are groups and $\phi:G_1\to G_2$ is a surjective homomorphism then given a finitely supported generating probability measure $\mu$ on $G_1$ we define a finitely supported generating probability measure $\phi(\mu)$ on $G_2$ by setting $\phi(\mu)(g)=\sum_{\overline g\in\phi^{-1}(g)}\mu(\overline g)$. Note that if $\mu$ is symmetric then so is $\phi(\mu)$.
\begin{lemma}\label{lem:harm.pullback}
Let $G_1$ be a group with a finitely supported generating probability measure $\mu$. Suppose that $\phi:G_1\to G_2$ is a surjective homomorphism, and that $f:G_2\to\R$. Then $f\circ\phi$ is harmonic with respect to $\mu$ if and only if $f$ is harmonic with respect to $\phi(\mu)$.
\end{lemma}
\begin{proof}
Given an arbitrary $g_2\in G_2$, the surjectivity of $\phi$ implies that there exists $g_1\in G_1$ such that $\phi(g_1)=g_2$. On the other hand, given an arbitrary $g_1\in G_1$, we may simply define $g_2\in G_2$ by $g_2=\phi(g_1)$. In either case, $f(g_2)=f\circ\phi(g_1)$ and $\E_{s\in\phi(S)}f(g_2s)=\E_{s\in S}f(g_2\phi(s))=\E_{s\in S}f\circ\phi(g_1s)$, from which the lemma follows easily.
\end{proof}
Given a subset $A$ of a graph $\Gamma$, or of a group $G$ with a finitely supported generating probability measure $\mu$, we define the \emph{neighbourhood} $A^+$ of $A$ to be the set $A^+=\{x\in\Gamma:d(x,A)\le1\}$, the \emph{interior} $A^\circ$ of $A$ to be the set $A^\circ=\{x\in A:\{x\}^+\subset A\}$, the \emph{inner boundary} $\partial^-A$ of $A$ to be the set $\partial^-A=A\backslash A^\circ$, and the \emph{outer boundary} $\partial^+A$ of $A$ to be the set $\partial^+A=A^+\backslash A$.

Let $\Gamma$ be a locally finite weighted graph, or a group with a finitely supported generating probability measure $\mu$. Let $A$ be a subset of $\Gamma$, and let $D$ be a subset of $\Gamma$ containing $A^+$. Then we say that a function $h:D\to\R$ is \emph{harmonic on $A$} if we have $\Delta h(x)=0$ for each $x\in A$.

The following is an immediate consequence of the definition of harmonicity.
\begin{lemma}[Maximum principle]\label{lem:max.princ}
Let $\Gamma$ be a locally finite graph, or a group with a finitely supported generating probability measure $\mu$, and let $A$ be a connected subset of $\Gamma$. Suppose that $f:A^+\to\R$ is harmonic on $A$ and achieves a maximum on $A$. Then $f$ is constant.
\end{lemma}

Harmonic functions on graphs and groups are intimately connected to random walks. Given a graph $\Gamma$ and a vertex $x\in\Gamma$, the \emph{random walk} starting at $x$ is a sequence of $\Gamma$-valued random variables $X_0,X_1,X_2,\ldots$, with $X_0=x$ with probability $1$ and each subsequent $X_n$ chosen from among the neighbours of $X_{n-1}$ such that $X_n=y$ with probability $\omega_{X_{n-1}y}/\deg X_{n-1}$. Given a group $G$ with a finitely supported generating probability measure $\mu$, the random walk on the pair $(G,\mu)$ starting at $x\in G$ is a sequence of $G$-valued random variables $X_0,X_1,X_2,\ldots$, with $X_0=x$ with probability $1$ and each subsequent $X_n$ taking the value $X_{n-1}s$ with probability $\mu(s)$. We say that the random walk on $(G,\mu)$ is \emph{symmetric} if $\mu$ is symmetric.

Given an event $B$, we denote by $\Prob_x[\,B\,]$ the conditional probability $\Prob[\,B\,|\,X_0=x\,]$. Given another event $C$, we denote by $\Prob_x[\,B\,|\,C\,]$ the conditional probability $\Prob[\,B\,|\,C\text{ and }\{X_0=x\}\,]$. We use the conditional expectation notation $\E_x$ similarly.

If $A$ is a subset of $\Gamma$, we write $T_A:=\inf\{t:X_t\in A\}$, with $T_A=\infty$ if $X_t\notin A$ for all $t$. The random variable $T_A$ is often called a \emph{stopping time} for the random walk. If $A$ is the singleton $\{x\}$ then we abbreviate $T_x:=T_{\{x\}}$.

The next few results are standard; see, for example, \cite{pete}. 
\begin{lemma}[Harmonic functions are determined by their boundary values]\label{cor:boundary}
Let $\Gamma$ be a graph, or a group with a finitely supported generating probability measure, and let $A$ be a finite subset of $\Gamma$ with non-empty outer boundary. Let $f_0:\partial^+A\to\R$. Then the function $f:A^+\to\R$ defined by $f(x)=\E_x\left[f_0\left(X_{T_{\partial^+A}}\right)\right]$ is harmonic on $A$ and agrees with $f_0$ on $\partial^+A$, and is unique with respect to these two properties.
\end{lemma}
\begin{corollary}\label{cor:dominating.harm.fn}
Let $G$ be a group and let $A$ be a finite subset of $G$. Suppose that $f_1,f_2:A^+\to\R$ are harmonic on $A$, and that $f_1\ge f_2$ on $\partial^+A$. Then $f_1\ge f_2$ on the whole of $A^+$.
\end{corollary}
\begin{lemma}\label{lem:off-diag}
Let $x,y$ be vertices in a vertex-transitive weighted graph $\Gamma$. Then $\Prob_x[\,X_{2n}=y\,]\le\Prob_e[\,X_{2n}=e\,]$.
\end{lemma}
\begin{remark}
Lemma \ref{lem:off-diag} does not necessarily hold if $2n$ is replaced by $n$. For example, if $n$ is odd then in the Cayley graph $(\Z,\pm1)$ we have $\Prob_0[\,X_n=0\,]=0$.
\end{remark}
\begin{prop}\label{prop:vert.trans.symm}
Let $\Gamma$ be a locally finite vertex-transitive weighted graph, and let $x,y\in\Gamma$. Then for each $n$ we have $\Prob_x[\,T_y=n\,]=\Prob_y[\,T_x=n\,]$.
\end{prop}
\begin{remark}
Proposition \ref{prop:vert.trans.symm} is trivial for a Cayley graph. It does not necessarily hold in a regular graph that is not vertex transitive; see Figure \ref{fig:not.symm}.
\begin{figure}
\caption{A regular graph in which $\Prob_x[\,T_y<\infty\,]>\Prob_y[\,T_x<\infty\,]$.}\label{fig:not.symm}
\begin{center}
\begin{tikzpicture}
\draw (0,0) -- (1,0);
\draw (0,0) -- (1,1);
\draw (0,0) -- (1,-1);
\draw (1,1) -- (2,0);
\draw (1,-1) -- (2,0);
\draw (1,1) -- (1,-1);
\draw (2,0) node[below] {$x$} -- (3,0);
\draw (3,0) node[below] {$y$}-- (4,0.6);
\draw (3,0) -- (4,-0.6);
\draw (4,0.6) -- (4.8,0.8);
\draw (4,0.6) -- (4.4,1.3);
\draw (4,-0.6) -- (4.8,-0.8);
\draw (4,-0.6) -- (4.4,-1.3);
\draw[dashed] (4.8,0.8) -- (5.6,1);
\draw[dashed] (4.4,1.3) -- (4.8,2);
\draw[dashed] (4.8,-0.8) -- (5.6,-1);
\draw[dashed] (4.4,-1.3) -- (4.8,-2);
\filldraw (0,0) circle (2pt);
\filldraw (1,0) circle (2pt);
\filldraw (1,1) circle (2pt);
\filldraw (1,-1) circle (2pt);
\filldraw (2,0) circle (2pt);
\filldraw (3,0) circle (2pt);
\filldraw (4,0.6) circle (2pt);
\filldraw (4,-0.6) circle (2pt);
\filldraw (4.8,0.8) circle (2pt);
\filldraw (4.4,1.3) circle (2pt);
\filldraw (4.8,-0.8) circle (2pt);
\filldraw (4.4,-1.3) circle (2pt);
\end{tikzpicture}
\end{center}
\end{figure}
\end{remark}
Proposition \ref{prop:vert.trans.symm} seems to be well known -- see, for example, \cite[Proposition 2]{aldous} for a proof in the case of a finite graph -- but the author was unable to find in the literature a proof of it as stated, so we present one here. A key step is the following lemma.
\begin{lemma}\label{lem:reg.loops}
Let $\Gamma$ be a locally finite vertex-transitive weighted graph, and let $n\in\N$. Then for every $x,y\in\Gamma$ we have
\[
\begin{split}
\Prob_x[\,X_n=x,\text{and }X_i\ne y\text{ for all }i=1,\ldots,n-1\,]\qquad\qquad\qquad\qquad\\
\qquad\qquad\qquad\qquad=\Prob_y[\,X_n=y,\text{and }X_i\ne x\text{ for all }i=1,\ldots,n-1\,]
\end{split}
\]
\end{lemma}
\begin{proof}
If $n=0$ then the lemma is trivial, so by induction we may fix $n>0$ and assume that
\begin{align*}
&\Prob_x[\,X_r=x,\text{and }X_i\ne y\text{ for all }i=1,\ldots,r-1\,]\\
&\qquad\qquad=\Prob_y[\,X_r=y,\text{and }X_i\ne x\text{ for all }i=1,\ldots,r-1\,]\\
&\qquad\qquad=u_r,
\end{align*}
say, for every $r<n$. Moreover, since $\Gamma$ is regular, if $z_0,\ldots,z_r$ is a path from $x$ to $y$ then $\Prob_x[\,X_0=z_0,\ldots,X_r=z_r\,]=\Prob_y[\,X_0=z_r,\ldots,X_r=z_0\,]$. This means, in particular, that if $v_r(x,y)$ is the probability of moving from $x$ to $y$ in $r$ steps, without visiting either $x$ or $y$ in between, then
\begin{equation}\label{eq:v_r}
v_r(x,y)=v_r(y,x)=v_r,
\end{equation}
say, for every $r$.

It is immediate from the vertex transitivity of $\Gamma$ that we have
\begin{equation}\label{eq:walk.reg}
\Prob_x[\,X_n=x\,]=\Prob_y[\,X_n=y\,],
\end{equation}
and so it suffices to show that we have
\begin{equation}\label{eq:symm.negation}
\begin{split}
\Prob_x[\,X_n=x,\,\text{and }X_i=y\text{ for some }i=1,\ldots,n-1\,]\qquad\qquad\qquad\qquad\\
\qquad\qquad\qquad=\Prob_y[\,X_n=y,\,\text{and }X_i=x\text{ for some }i=1,\ldots,n-1\,].
\end{split}
\end{equation}
Given $k\ge1$ and a sequence $0\le a_1<b_1\le a_2<b_2\le\ldots\le a_k<b_k\le n$ of integers, define the event $L_{x,y}(n;k;a_1,\ldots,a_k;b_1,\ldots,b_k)$ to be the event that $X_0=X_n=x$ and, if $0=t_1<\ldots<t_l=n$ are all the times $t$ at which $X_t\in\{x,y\}$ and we set $A=\{t_i:X_{t_i}\ne X_{t_{i+1}}\}$ and $B=\{t_{i+1}:X_{t_i}\ne X_{t_{i+1}}\}$, then we have $A=\{a_1,\ldots,a_k\}$ and $B=\{b_1,\ldots,b_k\}$. Setting $a_{k+1}=n$ and $b_0=0$ for notational convenience, we have
\begin{align}
\Prob_x[\,L_{x,y}(n;k;a_1,\ldots,a_k;b_1,\ldots,b_k)\,]&=\prod_{i=0}^ku_{a_{i+1}-b_i}\prod_{j=1}^kv_{b_j-a_j}\nonumber\\
&=\Prob_y[\,L_{y,x}(n;k;a_1,\ldots,a_k;b_1,\ldots,b_k)\,].\label{eq:Lxy=Lyx}
\end{align}
However, the event $\{\,X_0=X_n=x,\,\text{and }X_i=y\text{ for some }i=1,\ldots,n-1\,\}$ is precisely the disjoint union of all events $L_{x,y}(n;k;a_1,\ldots,a_k;b_1,\ldots,b_k)$ with $k\ge1$, and so (\ref{eq:symm.negation}) follows immediately from (\ref{eq:Lxy=Lyx}). The lemma is then immediate from (\ref{eq:walk.reg}) and (\ref{eq:symm.negation}).
\end{proof}
\begin{proof}[Proof of Proposition \ref{prop:vert.trans.symm}]
We prove the more precise statement that
\[
\begin{split}
\Prob_x[\,T_y=n\text{ and }\max\{t<n:X_t=x\}=r\,]\qquad\qquad\qquad\\
\qquad\qquad=\Prob_y[\,T_x=n\text{ and }\max\{t<n:X_t=x\}=r\,]
\end{split}
\]
for every $r\ge0$. Indeed, this follows readily from Lemma \ref{lem:reg.loops} and (\ref{eq:v_r}), and the observation that
\[
\begin{split}
\Prob_x[\,T_y=n\text{ and }\max\{t<n:X_t=x\}\,=\,r\,]\qquad\qquad\qquad\qquad\qquad\qquad\\
\qquad\qquad\qquad=v_{n-r}(x,y)\Prob_x[\,X_r=x,\text{and }X_i\ne y\text{ for all }i=1,\ldots,r-1\,].
\end{split}
\]
\end{proof}
\begin{remark}
The only properties of $\Gamma$ that we used in the proof of Proposition \ref{prop:vert.trans.symm} were its regularity and (\ref{eq:walk.reg}). These properties are satisfied, more generally, by \emph{walk-regular} (unweighted) graphs (see \cite{georgakopoulos,god.mck} for definitions and background). Proposition \ref{prop:vert.trans.symm} therefore also holds in walk-regular unweighted graphs.
\end{remark}

A vertex $x$ of a graph, or a group with a finitely supported generating probability measure, is called \emph{recurrent} for the random walk on the graph or group if $\Prob_x[\,T_x<\infty\,]=1$, and \emph{transient} for the random walk otherwise. In the case of a connected graph or a group this is independent of the choice of vertex, and so it makes sense to define the random walk on a connected graph, or on a group with a finitely supported generating probability measure, to be \emph{recurrent} if $\Prob_e[\,T_e<\infty\,]=1$, and \emph{transient} otherwise.

Write $R_x$ for the number of times the random walk visits the vertex $x$. Note that in the case of a transient random walk the variable $R_e$ has a geometric distribution under the probability measure $\Prob_e$, from which the following well-known fact easily follows.
\begin{lemma}\label{lem:no.visits}
The random walk on a connected graph, or on a group with a finitely supported generating probability measure, is transient if and only if $\E_e[R_e]<\infty$.
\end{lemma}

In the case of a group, if we require probability measures to be symmetric then recurrence or transience of the random walk is even independent of the choice of finitely supported generating probability measure \cite[Proposition 4.2]{woess}. It therefore makes sense simply to define a finitely generated group to be \emph{recurrent} if some symmetric random walk on it is recurrent, and \emph{transient} otherwise.

N. Varopoulos has characterised those groups that are recurrent.
\begin{prop}[Varopoulos \cite{varopoulos,woess}]\label{prop:varopoulos}
Let $G$ be a group with a symmetric, finitely supported generating probability measure $\mu$. Then the random walk on $(G,\mu)$ is recurrent if and only if $G$ is finite or has a finite-index subgroup isomorphic to $\Z$ or $\Z^2$.
\end{prop}

We close this section by recording the following standard but repeatedly useful reduction.
\begin{lemma}\label{lem:normal}Let $G$ be a group and let $H$ be a finite-index subgroup of $G$. Then there exists a finite-index subgroup $H'<H$ that is normal in $G$.
\end{lemma}
\begin{proof}
It is easy to verify that the subgroup $H'=\bigcap_{gH\in G/H}gHg^{-1}$ is well defined, normal and of finite index in $G$.
\end{proof}
\section{A Garden of Eden theorem}\label{sec:GofE}
In this section we prove Theorem \ref{thm:GofE.neighbourly}. Throughout this section, we write $e$ for an arbitrary distinguished vertex of the graph $\Gamma$ under consideration, and write $B(n)=B_e(n)$ for the ball of radius $n$ about $e$.
\begin{lemma}[Ceccherini-Silberstein--Coornaert]\label{lem:cs-c.closure}
Let $\Gamma$ be a connected, locally finite graph and let $\tau:V^\Gamma\to V^\Gamma$ be a locally specifiable linear map. Suppose that $f:\Gamma\to V$ is such that for every $n$ there is a function $v_n:\Gamma\to V$ such that $\tau(v_n)$ and $f$ agree on the ball $B(n)$. Then there is a function $w:\Gamma\to V$ such that $f=\tau(w)$.
\end{lemma}
\begin{proof}[Proof \textup{\cite[Lemma 3.1]{GofE}}]
For each $n\ge2$, denote by $\tau_n$ the linear map $V^{B(n)}\to V^{B(n-1)}$ induced by $\tau$, and define $L_n$ to be the affine subspace of $V^{B(n)}$ given by $L_n=\tau_n^{-1}(f|_{B(n-1)})$. Note in particular that $v_{n-1}|_{B(n)}\in L_n$, so that $L_n$ is non-empty.

For $n\le m$, the restriction map $V^{B(m)}\to V^{B(n)}$ induces an affine map $\pi_{n,m}:L_m\to L_n$, and so we may define an affine subspace $K_{n,m}\subset L_n$ by $K_{n,m}=\pi_{n,m}(L_m)$. Since \begin{equation}\label{eq:composition.pi}
\pi_{n_1,n_3}=\pi_{n_1,n_2}\circ\pi_{n_2,n_3}
\end{equation}
whenever $n_1\le n_2\le n_3$, for any fixed $n$ we have $K_{n,n}\supset K_{n,n+1}\supset K_{n,n+2}\supset\ldots$, and so the sequence $K_{n,n}, K_{n,n+1}, K_{n,n+2},\ldots$ is a decreasing sequence of non-empty finite-dimensional affine subspaces. This sequence therefore stabilises at some non-empty affine subspace $J_n$ of $L_n$. The identity (\ref{eq:composition.pi}) also implies that whenever $n\le n'\le m$ we have $\pi_{n,n'}(K_{n',m})\subset K_{n,m}$, and so by taking $m$ sufficiently large we see in particular that $\pi_{n,n'}(J_{n'})\subset J_n$. We claim that in fact
\begin{equation}\label{eq:rho.surj}
\pi_{n,n'}(J_{n'})=J_n.
\end{equation}
Indeed, given $u\in J_n$, let $m$ be sufficiently large that $J_n=K_{n,m}$ and $J_{n'}=K_{n',m}$. By definition of $K_{n,m}$, there is some $v\in L_m$ such that $u=\pi_{n,m}(v)$, and then yet another application of (\ref{eq:composition.pi}) then shows that
\begin{equation}\label{eq:u.in.image}
u=\pi_{n,n'}(\pi_{n',m}(v)).
\end{equation}
However, $\pi_{n',m}(v)\in K_{n',m}=J_{n'}$ by definition of $K_{n',m}$, and so (\ref{eq:u.in.image}) implies that $u\in\pi_{n,n'}(J_{n'})$. Since, $u\in J_n$ was arbitrary, this proves (\ref{eq:rho.surj}), as claimed.

We now construct recursively a sequence of functions $w_n\in J_n$, $n\in\N$, as follows. Initially, choose an arbitrary function $w_1\in J_1$. Then, given $w_n\in J_n$, choose $w_{n+1}$ arbitrarily from the set $\pi_{n,n+1}^{-1}(w_n)\subset J_{n+1}$, which is non-empty by (\ref{eq:rho.surj}). Since $w_{n+1}$ and $w_n$ agree on $B(n)$, there exists $w\in V^\Gamma$ such that $w|_{B(n)}=w_n$ for every $n$. However, $\tau(w)|_{B(n-1)}=\tau_n(w_n)=f|_{B(n-1)}$ for every $n$ by construction, and so $\tau(w)=f$.
\end{proof}
\begin{proof}[Proof of Theorem \ref{thm:GofE.neighbourly}]
A locally specifiable map is pre-injective on $\Gamma$ if and only if it is pre-injective on every connected component of $\Gamma$, and surjective on $\Gamma$ if and only if it is surjective on every connected component of $\Gamma$, and so we may assume that $\Gamma$ is connected. This is essentially the same as a reduction to the countable case made by Ceccherini-Silberstein and Coornaert in their original proof of Theorem \ref{thm:GofE} \cite{ind.rest.lca}.

We first prove that surjectivity of $\tau$ implies pre-injectivity of $\tau'$. Given $v,w\in V=\K^r$, write $v\cdot w=\sum_{i=1}^rv_iw_i$, and given $f_1\in V^\Gamma_0$ and $f_2\in V^\Gamma$ write $f_1\cdot f_2=\sum_{x\in\Gamma}(f_1(x)\cdot f_2(x))$. Then if $\tau$ is surjective and $\varphi\in V^\Gamma_0$, we have
\begin{align*}
\tau'(\varphi)=0 & \Rightarrow \tau'(\varphi)\cdot f=0\text{ for every $f\in V^\Gamma$}\\
                          & \Rightarrow \varphi\cdot\tau(f)=0\text{ for every $f\in V^\Gamma$}\\
                          & \Rightarrow \varphi=0
\end{align*}
by surjectivity of $\tau$, and so $\tau'$ is pre-injective.

We now prove the harder direction, namely that pre-injectivity of $\tau'$ implies surjectivity of $\tau$. Lemma \ref{lem:cs-c.closure} means that in order to prove that $\tau$ is surjective it suffices to show that the linear map $\tau_n:V^{B(n)}\to V^{B(n-1)}$ induced by $\tau$ is surjective. Since $\tau_n$ is a map between finite-dimensional spaces, it therefore suffices to show that its dual $\tau_n^\ast:V^{B(n-1)}\to V^{B(n)}$ is injective. However, the matrix of $\tau_n^\ast$ is precisely $\tau'$ restricted to $V^{B(n-1)}$ in domain and $V^{B(n)}$ in range, and so pre-injectectivity of $\tau'$ implies injectivity of $\tau_n^\ast$, which in turn implies surjectivity of $\tau_n$, as required.
\end{proof}
\section{Transpose-harmonic functions and surjectivity of Laplacians}\label{sec:Lap.surj}
In this section we prove Propositions \ref{prop:CSC.Lap.surj} and \ref{prop:grp.Lap.surj}. The proofs essentially consist of a fairly direct applications of Theorem \ref{thm:GofE.neighbourly}.
\begin{definition}[Transpose-harmonic function]Given a Laplacian $\Delta$ on a graph or a group $\Gamma$, we denote by $\Delta'$ the transpose of $\Delta$, and say that a function $h:\Gamma\to\R$ is \emph{transpose harmonic} if $\Delta'h=0$.
\end{definition}

If $\Delta=\Delta_\mu$ is the Laplacian on a group defined by a finitely supported generating probability measure $\mu$ then, writing $\mu'$ for the finitely supported generating probability measure defined by $\mu'(g)=\mu(g^{-1})$ we have
\begin{equation}\label{eq:Delta'mu'}
(\Delta_\mu)'=\Delta_{\mu'}.
\end{equation}
In the case of the Laplacian on a weighted graph, on the other hand, we have the following.
\begin{lemma}\label{lem:transp.harm}
Let $\Delta$ be the Laplacian on a locally finite weighted graph $\Gamma$, and let $f:\Gamma\to\R$ be a function. Then for each $x\in\Gamma$ we have
\[
\Delta'f(x)=f(x)-\sum_{y\sim x}\frac{\omega_{xy}f(y)}{\deg y}
\]
In particular, $f$ is transpose harmonic at $x$ if and only if the function $\hat f:\Gamma\to\R$ defined by
\[
\hat f(y)=\frac{f(y)}{\deg y}
\]
is harmonic at $x$.
\end{lemma}
\begin{proof}
The matrix of $\Delta$ is not hard to describe. In the row corresponding to the point $x$, the matrix has $1$ in the column corresponding to $x$; it has $-\omega_{xy}/\deg x$ in each column corresponding to a neighbour $y$ of $x$; and every other entry is zero. The $x$ row in the matrix of $\Delta'$ therefore has $1$ in the column corresponding to $x$; for each neighbour $y$ of $x$ it has $-\omega_{xy}/\deg y$ in the column corresponding to $y$; and every other entry is zero. The desired result follows immediately.
\end{proof}
\begin{proof}[Proof of Propositions \ref{prop:CSC.Lap.surj} and \ref{prop:grp.Lap.surj}]
In each case, Theorem \ref{thm:GofE.neighbourly} shows that it is sufficient to prove that a finitely supported transpose-harmonic function is identically zero.

In the case of the Laplacian on an infinite, connected, locally finite weighted graph (as in Proposition \ref{prop:CSC.Lap.surj}), Lemma \ref{lem:transp.harm} implies that the required statement is equivalent to showing that a finitely supported harmonic function is identically zero, since $\hat f(x)=0$ if and only if $f(x)=0$.

In the case of the Laplacian defined by a finitely supported generating probability measure $\mu$ (as in Proposition \ref{prop:grp.Lap.surj}), (\ref{eq:Delta'mu'}) implies that the required statement is equivalent to showing that a finitely supported $\mu'$-harmonic function is identically zero.
%

In each case, the required statement follows from the maximum principle (Lemma \ref{lem:max.princ}), and so the propositions are both proved.
\end{proof}
\begin{remarks}\label{rem:Lap.surj.cases}
The proof just presented is modelled on the amenable case of the proof of \cite[Theorem 1.1]{Lap.surj}, which is Proposition \ref{prop:grp.Lap.surj} in the special case that $\mu$ is uniform on a finite symmetric generating set. The proof of \cite[Theorem 1.1]{Lap.surj} in the amenable case uses Theorem \ref{thm:GofE} in place of Theorem \ref{thm:GofE.neighbourly}. The fact that Theorem \ref{thm:GofE} does not necessarily hold in non-amenable groups forces the authors to use a different argument in that case, in particular relying on a spectral criterion for amenability of finitely generated groups due to Kesten and Day. Our use of Theorem \ref{thm:GofE.neighbourly} allows us to avoid this complication.

Our arguments would also prove Proposition \ref{prop:CSC.Lap.surj} for an \emph{asymetrically weighted graph}, which is to say if we were to drop the assumption that $\omega_{xy}=\omega_{yx}$, provided it satisfied $\sum_{y\sim x}\omega_{xy}=\sum_{y\sim x}\omega_{yx}$ for every $x$.
\end{remarks}
%
%
%
%
%
%
%
%
%
\section{A duality result for harmonic functions}\label{sec:duality.result}
The aim of this section is to prove the following result.
\begin{prop}[Duality result for harmonic functions]\label{prop:dual.harmonic}Let $\Gamma$ be an infinite, connected, locally finite weighted graph, and let $X$ be a finite subset of $\Gamma$. Then the following statements are equivalent.
\begin{enumerate}
\renewcommand{\labelenumi}{(\arabic{enumi})}
\item Every function $f:X\to\R$ extends to a harmonic function on all of $\Gamma$.
\item There is no non-zero finitely supported function on $\Gamma$ that is harmonic on $\Gamma\backslash X$.
\end{enumerate}
\end{prop}
\begin{remark}
Proposition \ref{prop:dual.harmonic} fails in a finite graph, or a graph with a finite connected component, since statement (2) never holds in a finite graph, but statement (1) holds in an arbitrary graph when $X$ is a singleton. See Remark \ref{rem:false.for.finite} for details on where the proof breaks down.
\end{remark}
Given a subset $Y$ of $\Gamma$, we denote by $\R^\Gamma_Y$ the subspace of $\R^\Gamma$ consisting of those functions supported on $Y$. Proposition \ref{prop:dual.harmonic} then follows from combining the following two lemmas with Proposition \ref{prop:CSC.Lap.surj}, which implies that $\Delta(\R^\Gamma)=\R^\Gamma$.
\begin{lemma}\label{lem:dual.harm.1}
Let $\Gamma$ be a locally finite weighted graph, and let $X\subset\Gamma$ be a finite set. Then the following statements are equivalent.
\begin{enumerate}
\renewcommand{\labelenumi}{(\arabic{enumi})}
\item We have $\Delta(\R^\Gamma_{\Gamma\backslash X})=\R^\Gamma$.
\item There is no non-zero finitely supported function on $\Gamma$ that is harmonic on $\Gamma\backslash X$.
\end{enumerate}
\end{lemma}
\begin{lemma}\label{lem:dual.harm.2}
Let $\Gamma$ be a locally finite weighted graph, and let $X\subset\Gamma$ be a finite set. Then the following statements are equivalent.
\begin{enumerate}
\renewcommand{\labelenumi}{(\arabic{enumi})}
\item We have $\Delta(\R^\Gamma_{\Gamma\backslash X})=\Delta(\R^\Gamma)$.
\item Every function $f:X\to\R$ extends to a harmonic function on all of $\Gamma$.
\end{enumerate}
\end{lemma}
\begin{proof}[Proof of Lemma \ref{lem:dual.harm.1}]
First note that by Lemma \ref{lem:transp.harm} and the fact that for every function $f:\Gamma\to\R$ we have $\hat f(x)=0$ if and only if $f(x)=0$, statement (2) of Lemma \ref{lem:dual.harm.1} is equivalent to the following statement.
\begin{itemize}
\item[(2$'$)]There is no non-zero finitely supported function on $\Gamma$ that is transpose harmonic on $\Gamma\backslash X$.
\end{itemize}
Abusing notation slightly, we identify the operator $\Delta$ with its (possibly infinite) matrix. Statement (1) of the lemma is then equivalent to saying that the matrix $\Delta_{\Gamma\backslash X}$ obtained by replacing the columns of $\Delta$ corresponding to the elements of $X$ with columns of zeros is surjective.

Statement (2$'$), on the other hand, means that if $f\in\R^\Gamma_0$ is non-zero then $\Delta'(f)$ cannot be zero on $\Gamma\backslash X$. Put another way, this says that even if we replace the rows of $\Delta'$ corresponding to the elements of $X$ with columns of zeros then $\Delta'$ will be pre-injective.

However, $\Delta'$ with the rows corresponding to $X$ replaced by zeros is equal to the transpose of $\Delta_{\Gamma\backslash X}$. Replacing some entries of $\Delta$ by zeros does not change the fact that it is a locally specifiable map, and so the equivalence of (1) and (2$'$) therefore follows from Theorem \ref{thm:GofE.neighbourly}.
\end{proof}
\begin{proof}[Proof of Lemma \ref{lem:dual.harm.2}]
We first prove that (1) implies (2). Let $f:X\to\R$ be arbitrary, and define $\overline{f}$ to be the function on $\Gamma$ that agrees with $f$ on $X$ and takes the value $0$ elsewhere. By (1) we can find a function $h$ supported on $\Gamma\backslash X$ such that $\Delta(h)=\Delta(-\overline{f})$. The function $h+\overline{f}$ is then a harmonic extension of $f$, and so (2) is proved.

Conversely, note that in order to prove (1) it suffices to prove that for every $x\in X$ the function $\Delta(\delta_x)$ lies in the space $\Delta(\R^\Gamma_{\Gamma\backslash X})$. However, if we assume (2) then in particular we have a harmonic extension $h$ of the function $f:X\to\R$ taking the value $1$ at $x$ and $0$ on $X\backslash\{x\}$, and it immediately follows that $\Delta(\delta_x)=\Delta(-h|_{\Gamma\backslash X})$.
\end{proof}
\begin{remark}\label{rem:false.for.finite} In the case that $\Gamma$ has a finite connected component, Proposition \ref{prop:CSC.Lap.surj} no longer holds, and so Lemmas \ref{lem:dual.harm.1} and \ref{lem:dual.harm.2} no longer combine to prove Proposition \ref{prop:dual.harmonic}.
\end{remark}
\section{Existence of non-constant harmonic functions on graphs}\label{sec:exist.of.harm.fn}
In this section we use Proposition \ref{prop:dual.harmonic} to prove the following result, which generalises Proposition \ref{prop:trofimov} in the transient case.
\begin{prop}\label{prop:harm.exists}
Let $\Gamma$ be a locally finite vertex-transitive weighted graph, and suppose that the random walk on $\Gamma$ is transient. Suppose that $K$ is finitely generated subgroup of $\Aut\Gamma$ such that the orbit $Ke$ is infinite. Then there exists a harmonic function on $\Gamma$ that is not constant on $Ke$.
\end{prop}
\begin{remarks}
Proposition \ref{prop:harm.exists} applies in particular to groups with symmetric, finitely supported generating probability measures, since they can be realised as vertex-transitive weighted graphs by considering their weighted Cayley graphs.

Proposition \ref{prop:harm.exists} does not necessarily hold if $K$ has finite orbits. For example, if $G=\Z^3\oplus\Z/2\Z$ and $S=\{(\pm e_1,0),(\pm e_1,1),(\pm e_2,0),(\pm e_2,1),(\pm e_3,0),(\pm e_3,1),(0,1)\}$ and $\Gamma$ is the Cayley graph $(G,S)$, then every harmonic function on $G$ is constant on the orbits of $\Z/2\Z$.
\end{remarks}
Let us note how Proposition \ref{prop:harm.exists} implies the transient case of Proposition \ref{prop:trofimov}. Proposition \ref{prop:trofimov} is trivial when $\Gamma$ is not connected; when $\Gamma$ is connected and transient it follows immediately from Proposition \ref{prop:harm.exists} and the following lemma.
\begin{lemma}\label{lem:fin.gen}
Let $\Gamma$ be a connected, locally finite, vertex-transitive weighted graph. Then there is a finitely generated subgroup $G<\Aut\Gamma$ that is transitive.
\end{lemma}
\begin{proof}
Let $e\in\Gamma$. By the transitivity of $\Aut\Gamma$, for each neighbour $y$ of $e$ there is an automorphism $g_y$ of $\Gamma$ such that $g_ye=y$. We claim that $G:=\langle\, g_y:y\sim e\,\rangle$ is transitive; since $\Gamma$ is locally finite, this is sufficient to prove the lemma.

Since $\Gamma$ is connected, it suffices to show that if $z\in Ge$ and $x\sim z$ then $x\in Ge$. To see this, note that for $z\in Ge$ there exists $h\in G$ such that $e=hz$. However, this means that we have $hx\sim e$, and so $x=h^{-1}g_{hx}e\in Ge$, as desired, and the lemma is proved.
\end{proof}

We also recover from Proposition \ref{prop:harm.exists} the following well-known fact.
\begin{corollary}\label{cor:harm.exists}
Let $G$ be an infinite group with a symmetric, finitely supported generating probability measure $\mu$. Then $(G,\mu)$ admits a non-constant harmonic function.
\end{corollary}
\begin{proof}If the random walk on $(G,\mu)$ is transient then the corollary follows immediately from Proposition \ref{prop:harm.exists} if we let $G$ act on its own Cayley graph by left multiplication and take $K=G$. If the random walk is recurrent then Proposition \ref{prop:varopoulos} implies that $G$ has either $\Z$ or $\Z^2$ as a finite-index subgroup, in which case the corollary follows from \cite{salvatori} or from Lemma \ref{lem:H1G}, below.
\end{proof}
\bigskip
\noindent For the remainder of this section we are concerned with proving Proposition \ref{prop:harm.exists}. Throughout, $\Gamma$ is a locally finite vertex-transitive weighted graph with distinguished vertex $e$.

By Proposition \ref{prop:dual.harmonic}, in order to prove Proposition \ref{prop:harm.exists} in the connected case it suffices to find two points $x,y\in Ke$ with the property that there is no non-zero finitely supported function on $\Gamma$ that is harmonic except at $x,y$. The following result gives a necessary condition for the existence of such a function.
\begin{lemma}\label{lem:crit.for.no.finite.sup.harm}
Let $x,y\in\Gamma$ and suppose that there exists a finitely-supported non-zero function $f:\Gamma\to\R$ that is harmonic except at $x$ and $y$. Then there exists some $N>0$ such that the conditional probability $\Prob_g[\,T_x<T_y\,|\,\min\{T_x,T_y\}<\infty\,]$ is independent of $g$ for $d(e,g)\ge N$.
\end{lemma}
\begin{proof}
Since $f$ is finitely supported, there is some $N>d(e,x),d(e,y)$ such that $f(g)=0$ whenever $d(e,g)\ge N$. We prove that the lemma holds with this $N$.

For $M\in\N$ we denote by $B(M)=B_e(M)$ the ball of radius $M$ about the vertex $e$, and by $\tau_M$ the quantity $\tau_M=\min\{t:X_t\in(G\backslash B(M))\cup\{x,y\}\}$, where $X_0,X_1,\ldots$ is, as usual, the random walk on $\Gamma$. By countable additivity of $\Prob$, for $g\in B(M)$ we have
\begin{align}
&\Prob_g[\,X_{\tau_M}=x\,|\,X_{\tau_M}\in\{x,y\}\,]\to\Prob_g[\,T_x<T_y\,|\,\min\{T_x,T_y\}<\infty\,],\label{al:prob.converge}\\
&\Prob_g[\,X_{\tau_M}=y\,|\,X_{\tau_M}\in\{x,y\}\,]\to\Prob_g[\,T_y<T_x\,|\,\min\{T_x,T_y\}<\infty\,]\label{al:prob.converge'}
\end{align}
as $M\to\infty$.

Let $M\ge N$. By Lemma \ref{cor:boundary}, there is a unique function $f_M:B(M+1)\to\R$ that is harmonic on $B(M)\backslash\{x,y\}$ and satisfies the following conditions:
\begin{align}
&f_M(x)=f(x);\label{al:1}\\
&f_M(y)=f(y);\label{al:2}\\
&f_M(z)=0\text{ for }z\notin B(M);\label{al:3}
\end{align}
indeed, Lemma \ref{cor:boundary} implies that
\begin{equation}\label{eq:f_M.as.prob}
f_M(g)=f(x)\cdot\Prob_g[\,X_{\tau_M}=x\,]+f(y)\cdot\Prob_g[\,X_{\tau_M}=y\,]
\end{equation}
for $g\in B(M)$.

The restriction $f|_{B(M+1)}$ is of course harmonic on $B(M)\backslash\{x,y\}$, and trivially satisfies (\ref{al:1}) and (\ref{al:2}); by the definitions of $N$ and $M$ it also satisfies condition (\ref{al:3}), and so by the uniqueness of $f_M$ it follows that
\begin{equation}\label{eq:f_M=f|M}
f_M=f|_{B(M+1)}.
\end{equation}
By the maximum principle (Lemma \ref{lem:max.princ}), and since $f$ is not identically zero, $f$ must be non-zero at at least one of $x$ and $y$; without loss of generality we may therefore assume that $f(x)\ne0$.
If $N\le|g|\le M$ then (\ref{eq:f_M=f|M}) and the definition of $N$ together imply that $f_M(g)=0$, and so (\ref{eq:f_M.as.prob}) implies that
\[
\frac{\Prob_g[\,X_{\tau_M}=x\,]}{\Prob_g[\,X_{\tau_M}=y\,]}=-\frac{f(y)}{f(x)},
\]
and hence that
\[
\frac{\Prob_g[\,X_{\tau_M}=x\,|\,X_{\tau_M}\in\{x,y\}\,]}{\Prob_g[\,X_{\tau_M}=y\,|\,X_{\tau_M}\in\{x,y\}\,]}=-\frac{f(y)}{f(x)}.
\]
Letting $M\to\infty$, we therefore see from (\ref{al:prob.converge}) and (\ref{al:prob.converge'}) that
\begin{equation}\label{eq:determines.p.uniquely}
\frac{\Prob_g[\,T_x<T_y\,|\,\min\{T_x,T_y\}<\infty\,]}{\Prob_g[\,T_y<T_x\,|\,\min\{T_x,T_y\}<\infty\,]}=-\frac{f(y)}{f(x)}.
\end{equation}
Since the numerator and denominator of the left-hand side of (\ref{eq:determines.p.uniquely}) always sum to 1, this determines $\Prob_g[\,T_x<T_y\,|\,\min\{T_x,T_y\}<\infty\,]$ uniquely and independently of $g$, and so the lemma is proved.
\end{proof}
The following lemma proves the intuitively reasonable result that if the random walk is more likely to hit $x$ than $y$ eventually, then it is also more likely to hit $x$ first.
\begin{lemma}\label{lem:hits.one.first}
If $x,y\in\Gamma$ satisfy
\begin{equation}\label{eq:hypoth}
\Prob_e[\,T_x<\infty\,]>\Prob_e[\,T_y<\infty\,]
\end{equation}
then they also satisfy
\begin{equation}\label{eq:concl}
\Prob_e[\,T_x<T_y\,|\,\min\{T_x,T_y\}<\infty\,]>1/2.
\end{equation}
If the random walk on $\Gamma$ is transient then (\ref{eq:hypoth}) and (\ref{eq:concl}) are equivalent.
\end{lemma}
\begin{remark}
The conditions (\ref{eq:hypoth}) and (\ref{eq:concl}) are not necessarily equivalent in a vertex-transitive graph with a recurrent random walk, as can be seen by setting $e=0$, $x=1$ and $y=2$ in the Cayley graph $(\Z,\{\pm1\})$.
\end{remark}
\begin{proof}[Proof of Lemma \ref{lem:hits.one.first}]
Write $p(x,y)=\Prob_x[\,T_y<\infty\,]$, the probability that the random walk starting at $x$ hits $y$ eventually. If $\Prob_e[\,T_x<\infty\,]>\Prob_e[\,T_y<\infty\,]$ then this implies in particular that $\Prob_e[\,T_z<\infty\,]$ is not constant in $z$, which implies that the random walk is transient. We may therefore assume that the random walk is transient and prove that (\ref{eq:hypoth}) and (\ref{eq:concl}) are equivalent.

Write
\[
p(x)=\Prob_e[\,T_x<\infty\,|\,\min\{T_x,T_y\}<\infty\,], 
\]
\[
p(y)=\Prob_e[\,T_y<\infty\,|\,\min\{T_x,T_y\}<\infty\,],
\]
and note that condition (\ref{eq:hypoth}) is equivalent to $p(x)>p(y)$. Write $f(x)=\Prob_e[\,T_x<T_y\,|\,\min\{T_x,T_y\}<\infty\,]$ and $f(y)=\Prob_e[\,T_y<T_x\,|\,\min\{T_x,T_y\}<\infty\,]$. Condition (\ref{eq:concl}) is that $f(x)>1/2$, or equivalently that $f(x)>f(y)$, since $f(x)+f(y)=1$. However, we have $p(y)=f(y)+f(x)p(x,y)$, and by Proposition \ref{prop:vert.trans.symm} we have $p(x,y)=p(y,x)$, and hence $p(x)=f(x)+f(y)p(x,y)$. The equivalene of (\ref{eq:hypoth}) and (\ref{eq:concl}) therefore follows, since transience of the random walk and symmetry of $p$ together imply that $p(x,y)<1$.
\end{proof}
\begin{prop}\label{prop:trans.prob.to.zero}
If the random walk on $\Gamma$ is transient then $\Prob_x[\,T_y<\infty\,]\to0$ as $d(x,y)\to\infty$.
\end{prop}
\begin{proof}
It is clear that $\Prob_x[\,T_y<\infty\,]\le\sum_{n=0}^\infty\Prob_x[\,X_n=y\,]$. However, since $\Prob_x[\,X_n=y\,]=0$ for $n<d(x,y)$, we in fact have the stronger bound $\Prob_x[\,T_y<\infty\,]\le\sum_{n=d(x,y)}^\infty\Prob_x[\,X_n=y\,]$. If $n$ is even then we have $\Prob_x[\,X_n=y\,]\le\Prob_e[\,X_n=e\,]$ by Lemma \ref{lem:off-diag}. If $n$ is odd, on the other hand, then we have $\Prob_x[\,X_n=y\,]=\E_{s\in S}\Prob_{xs}[\,X_{n-1}=y\,]\le\Prob_e[\,X_{n-1}=e\,]$, again by Lemma \ref{lem:off-diag}. Combining these last three inequalities shows that
\begin{equation}\label{eq:partial.sum.dist.2}
\Prob_x[\,T_y<\infty\,]\le2\sum_{\substack{n\ge d(x,y)-1 \\ n\text{ even}}}\Prob_e[\,X_n=e\,]
\end{equation}
Recall that $R_e$ is the number of times the random walks hits the vertex $e$. In particular, $R_e=\sum_{n=0}^\infty1_{\{X_n=e\}}$, and so by linearity of expectation we have $\E_e[R_e]=\sum_{n=0}^{\infty}\Prob_e[\,X_n=e\,]$. Lemma \ref{lem:no.visits} therefore implies that $\sum_{n=0}^{\infty}\Prob_e[\,X_n=e\,]<\infty$, which, combined with (\ref{eq:partial.sum.dist.2}), shows that $\Prob_x[\,T_y<\infty\,]\to0$ as $d(x,y)\to\infty$, as desired.
\end{proof}
\begin{proof}[Proof of Proposition \ref{prop:harm.exists}]
If the orbit $Ke$ has non-trivial intersection with two connected components of $\Gamma$ then the result follows by taking a function that takes the value $1$ on one of these components and $0$ elsewhere on $\Gamma$. We may therefore assume that $\Gamma$ is connected, and so by Proposition \ref{prop:dual.harmonic} it suffices to find two points $x,y\in Ke$ with the property that there is no non-zero finitely supported function on $\Gamma$ that is harmonic except at $x,y$.

We consider the following two cases.
\begin{enumerate}
\renewcommand{\labelenumi}{(\arabic{enumi})}
\item The subgroup $K$ contains an element $v$ such that the vertices $v^ne$ are all distinct for $n\in\N$.
\item For every element $u$ of the subgroup $K$ there is some $m$ such that $u^me=e$.
\end{enumerate}
In case (1), Proposition \ref{prop:trans.prob.to.zero} implies that $\Prob_e[\,T_{v^ne}<\infty\,]\to0$ and $\Prob_e[\,T_{v^{-n}e}<\infty\,]\to0$ as $n\to\infty$. This implies that there are infinite increasing sequences $n^+_1,n^+_2,n^+_3,\ldots$ and $n^-_1,n^-_2,n^-_3,\ldots$ such that $\Prob_e[\,T_{v^{n^+_i}e}<\infty\,]>\Prob_e[\,T_{v^{n^+_i}ve}<\infty\,]$ and $\Prob_e[\,T_{v^{-n^-_i}e}<\infty\,]<\Prob_e[\,T_{v^{-n^-_i}ve}<\infty\,]$, which by Lemma \ref{lem:hits.one.first} means that
\[
\Prob_{v^{-n^+_i}e}\left[\left.\,T_e>T_{ve}\,\right|\,\min\left\{T_e,T_{ve}\right\}<\infty\,\right]\,\,>\,\,1/2,
\]
\[
\Prob_{v^{n^-_i}e}\left[\left.\,T_e>T_{ve}\,\right|\,\min\left\{T_e,T_{ve}\right\}<\infty\,\right]\,\,<\,\,1/2.
\]
Since $v^{-n^+_i}e\to\infty$ and $v^{n^-_i}e\to\infty$, Lemma \ref{lem:crit.for.no.finite.sup.harm} therefore implies that there exists no finitely supported non-zero function on $\Gamma$ that is harmonic except at $e,ve$, and so the proposition is proved in case (1).

In case (2), let $R$ be a finite symmetric generating set for $K$. We claim that there are elements $x_1,x_2,\ldots\in K$ with $d(e,x_ne)\to\infty$ such that, for each $n$, there is some $r_n\in R$ such that $\Prob_e[\,T_{x_ne}<\infty\,]<\Prob_e[\,T_{x_nr_ne}<\infty\,]$. Indeed, for each $n=1,2,\ldots$, let $x_n$ be a point of minimal distance from the identity in the Cayley graph $(K,R)$ such that $\Prob_e[\,T_{x_ne}<\infty\,]<1/n$. Such a point always exists by Proposition \ref{prop:trans.prob.to.zero} and the assumption that the orbit $Ke$ is infinite, and by the regularity and local finiteness of $\Gamma$ we have
\begin{equation}\label{eq:x_n.to.infty}
d(e,x_ne)\to\infty
\end{equation}
as $n\to\infty$. By definition of $x_n$, and using (\ref{eq:x_n.to.infty}), for sufficiently large $n$ there is some $r_n\in R$ such that $\Prob_e[\,T_{x_nr_ne}<\infty\,]\ge1/n>\Prob_e[\,T_{x_ne}<\infty\,]$, as caimed.

By the finiteness of $R$, upon passing to a subsequence if necessary we may in fact assume that there is some $u\in R$ such that for each $n$ we have
\begin{equation}\label{eq:pxnu.bigger}
\Prob_e[\,T_{x_ne}<\infty\,]<\Prob_e[\,T_{x_nue}<\infty\,].
\end{equation}
We claim that there is no non-zero finitely supported function on $\Gamma$ that is harmonic except at $e,ue$.

As in case (1), condition (\ref{eq:pxnu.bigger}) and Lemma \ref{lem:hits.one.first} imply that
\[
\Prob_{x_n^{-1}e}\left[\left.\,T_e>T_{ue}\,\right|\,\min\left\{T_e,T_{ue}\right\}<\infty\,\right]\,\,<\,\,1/2;
\]
indeed, applying the automorphism $u^m$, we see that
\begin{equation}\label{eq:delayed.in.philly}
\Prob_{u^mx_n^{-1}e}\left[\left.\,T_{u^me}>T_{u^{m+1}e}\,\right|\,\min\left\{T_{u^me},T_{u^{m+1}e}\right\}<\infty\,\right]\,\,<\,\,1/2
\end{equation}
for every $m\in\N$. Moreover, (\ref{eq:x_n.to.infty}) implies that for each $m\in\N$ we have $d(u^{-m}e,x_n^{-1}e)\to\infty$ as $n\to\infty$, and so $d(e,u^mx_n^{-1}e)\to\infty$ as $n\to\infty$. If the claim is false, and there does exist some non-zero finitely supported function on $\Gamma$ that is harmonic except at $e,ue$, then translating this function by $u^m$ we see that there is also a function on $\Gamma$ that is harmonic except at $u^me,u^{m+1}e$. Combining (\ref{eq:x_n.to.infty}) and (\ref{eq:delayed.in.philly}) with Lemma \ref{lem:crit.for.no.finite.sup.harm} therefore implies that for each $m$ there is some $N_m>0$ such that
\[
\Prob_x\left[\,T_{u^me}>T_{u^{m+1}e}\,|\,\min\left\{T_{u^me},T_{u^{m+1}e}\right\}<\infty\,\right]\,\,<\,\,1/2
\]
for every $x\in\Gamma$ such that $d(e,x)\ge N_m$; since the orbit of $e$ under $u$ is finite we may assume that the $N_m$ are all equal to some $N>0$. Fixing some $x$ with $d(e,x)\ge N$ and applying Lemma \ref{lem:hits.one.first} once more, this means that
\[
\Prob_x[\,T_{u^me}<\infty\,]<\Prob_x[\,T_{u^{m+1}e}<\infty\,]
\]
for every $m\in\N$, which implies by induction that
\[
\Prob_x[\,T_e<\infty\,]<\Prob_x[\,T_{u^me}<\infty\,]
\]
for every $m\in\N$. This is impossible, however, since there is some $m\in\N$ such that $u^me=e$, and so it must have been the case that there was no non-zero finitely supported function on $\Gamma$ harmonic except at $e,ue$. This proves the claim, and hence the proposition in case (2).
\end{proof}
\section{Harmonic functions on virtually abelian groups}\label{sec:v.ab}
In this section we investigate spaces of harmonic functions on virtually abelian groups. The first purpose is to prove the easier direction of Theorem \ref{con:harm.fin.dim}, as follows.
\begin{prop}[Direct statement of Theorem \ref{con:harm.fin.dim}]\label{prop:direct}
Let $G$ be a group with a finite-index subgroup isomorphic to $(\Z,+)$, and let $\mu$ be a symmetric, finitely supported generating probability measure on $G$. Then $\dim H(G,\mu)<\infty$.
\end{prop}
The second is to note a characterisation of the space $H^1(G,\mu)$ of harmonic functions of linear growth on a virtually abelian group $G$ (see Lemma \ref{lem:H1G}, below).

Let $G$ be a group with a finite-index normal subgroup isomorphic to $\Z^d$, and let $\mu$ be a symmetric, finitely supported generating probability measure on $G$. Abbreviate $S:=\supp\mu$. Fix a right-transversal $T$ of $\Z^d$ containing the identity, which is to say a finite set $T$ such that each $g\in G$ can be expressed uniquely as $g=\zeta(g)\tau(g)$ with $\zeta(g)\in\Z^d$ and $\tau(g)\in T$. We write $\zeta_i(g)$ for the $i$th coordinate of $\zeta(g)$ with respect to the standard basis for $\Z^d$.
\begin{lemma}[\cite{alex,salvatori}]\label{lem:H1G}
For each $i=1,\ldots,d$ there is a function $\varphi_i:G\to\R$ that factors through $G/\Z^d$ such that the function $f_i:G\to\R$ given by $f_i(g)=\zeta_i(g)+\varphi_i(\tau(g))$ is harmonic on $(G,\mu)$. Moreover, $H^1(G,\mu)$ is spanned by the set $\{1,f_1,\ldots,f_d\}$.
\end{lemma}
\begin{proof}
The existence of the harmonic functions $f_i$ follows directly from \cite[Theorem 3.6]{salvatori}. The fact that $\{1,f_1,\ldots,f_d\}$ spans $H^1(G,\mu)$ is then precisely the linear-growth case of \cite[Theorem 1.12]{alex}; see also \cite{mpty} for a more elementary proof.
\end{proof}
\begin{lemma}\label{lem:boundary.in.virt.cyc}Let $d_{\Z^d}$ be the Cayley-graph distance on $\Z^d$ with respect to the standard generating set. Then there exists $M\in\N$ such that for every $g\in G$ and every $s\in S$ we have $d_{\Z^d}(\zeta(gs),\zeta(g))\le M$.
\end{lemma}
\begin{proof}Given $g\in G$ and $s\in S$, write $t=\tau(g)$, so that $gs=\zeta(g)ts=\zeta(g)\zeta(ts)\tau(ts)$. This implies, in particular, that $\zeta(gs)=\zeta(g)\zeta(ts)$, and so we may take $M$ to be the maximum over the (finite) set $\{|\zeta(ts)|_{\Z^d}:s\in S,t\in T\}$.
\end{proof}
\begin{proof}[Proof of Proposition \ref{prop:direct}]
Lemma \ref{lem:boundary.in.virt.cyc} implies that for each $n\in\Z$ we have $[-n,n]TS\subset [-n-M,n+M]T$. It follows that $([-n,n]T)^+\subset[-n-M,n+M]T$, and so $\partial^+([-n,n]T)$ has cardinality at most $2M|T|$. Lemma \ref{cor:boundary} therefore implies that the space of functions on $\partial^+([-n,n]T)$ that are harmonic on $[-n,n]T$ is of dimension at most $2M|T|$. However, $G=\bigcup_{n=1}^\infty [-n,n]T$, and so the space of harmonic functions on $G$ is also of dimension at most $2M|T|$.
\end{proof}
\begin{remark}
Taking $G=\Z$ and setting $\mu$ to be the uniform probability measure on $[-M,M]$ shows that the bound $2M|T|$ on the dimension of the space of harmonic functions in the proof of Proposition \ref{prop:direct} can be tight. In particular, the precise dimension depends on the measure $\mu$ as well as on the group $G$.
\end{remark}
\section{Positive harmonic functions on linear groups}\label{sec:linear}
If $G$ is a group and $\mu$ is a finitely supported generating probability measure then a \emph{positive harmonic function} on $(G,\mu)$ is a harmonic function $h:G\to\R$ that takes only positive values. G. Margulis \cite{margulis} showed that a nilpotent group admits no non-constant positive harmonic functions. More generally, we have the following result of W. Hebish and L. Saloff Coste.
\begin{prop}[Hebish--Saloff Coste \cite{heb-s-coste}]\label{prop:v.nilp.no.pos.harm}
Let $G$ be a virtually nilpotent group with a symmetric, finitely supported generating probability measure $\mu$. Then $(G,\mu)$ admits no non-constant positive harmonic functions.
\end{prop}
P. Bougerol and L. Elie show that for linear groups the converse is also true.
\begin{prop}[Bougerol--Elie \cite{boug-elie}]\label{prop:lin.pos.harm}
Let $G$ be a subgroup of $GL_d(\R)$ that is not virtually nilpotent, and let $\mu$ be a symmetric, finitely supported generating probability measure on $G$. Then $(G,\mu)$ admits a non-constant positive harmonic function.
\end{prop}
The purpose of this section is to show that, in that case, there are in fact many positive harmonic functions.
\begin{prop}\label{prop:pos.harm.hom}
Let $G$ be a group with a symmetric, finitely supported generating probability measure $\mu$, and suppose that $(G,\mu)$ admits at least one non-constant positive harmonic function. Then the set of positive harmonic functions on $(G,\mu)$ spans an infinite-dimensional space.
\end{prop}
The following is then immediate.
\begin{corollary}\label{cor:linear}
Let $G$ be a subgroup of $GL_d(\R)$ that is not virtually nilpotent, and let $\mu$ be a symmetric, finitely supported generating probability measure on $G$. Then the positive harmonic functions on $(G,\mu)$ span an infinite-dimensional space.
\end{corollary}
\begin{question}
Does an arbitrary non-virtually nilpotent group with a symmetric, finitely supported generating probability measure admit a non-constant positive harmonic function?
\end{question}
\bigskip

\noindent In proving Proposition \ref{prop:pos.harm.hom} we make use of the \emph{minimal Martin boundary} of $(G,\mu)$.
\begin{definition}[Minimal harmonic function]
Given a group $G$ with a finitely supported generating probability measure $\mu$, a \emph{minimal harmonic function} on $(G,\mu)$ is a positive harmonic function $f:G\to\R$ with the property that every other positive harmonic function $f':G\to\R$ satisfying $f'\le f$ is a constant multiple of $f$. A \emph{normed minimal harmonic function} $f:G\to\R$ is a minimal harmonic function satisfying $f(e)=1$.
\end{definition}
\begin{definition}[Minimal Martin boundary]
The \emph{minimal Martin boundary} $\Delta(G,\mu)$ of the pair $(G,\mu)$ is the compact closure, in the topology of pointwise convergence, of the set of normed minimal harmonic functions on $(G,\mu)$.
\end{definition}
Each positive harmonic function $f:G\to\R$ has a unique \emph{representing measure} $\nu_f$ on the minimal Martin boundary, which is to say a measure $\nu_f$ such that
\begin{equation}\label{eq:Martin}
f(x)=\int_\Delta h(x)d\nu_f(h)
\end{equation}
for every $x\in G$ (see \cite[\S0.3]{kai-ver} or \cite[\S7, p. 32]{woess}).
\begin{lemma}\label{lem:min.lin.indep}
The set of normed minimal harmonic functions on a group $G$ with respect to a finitely supported generating probability measure $\mu$ is linearly independent.
\end{lemma}
\begin{proof}
Suppose that $h_1,\ldots,h_r$ are distinct minimal harmonic functions and let $\alpha_1,\ldots,\alpha_r$ be such that $\sum_{i=1}^m\alpha_ih_i=0$. Without loss of generality we may assume that $\alpha_i\le0$ for $i\le k$, and that $\alpha_i\ge0$ for $i>k$, and so in fact we have $\sum_{i=1}^k(-\alpha_i)h_i=\sum_{i=k+1}^m\alpha_ih_i$. However, both the left-hand side and the right-hand side of this expression are non-negative harmonic functions, and so it follows from the uniquness of the representation (\ref{eq:Martin}) that the $\alpha_i$ are all zero.
\end{proof}
\begin{proof}[Proof of Proposition \ref{prop:pos.harm.hom}]
We prove the contrapositive. Suppose that the set of positive harmonic functions on $G$ does not span an infinite-dimensional space. By Lemma \ref{lem:min.lin.indep} this implies in particular that the set of normed minimal harmonic functions is finite, so we may enumerate them as $h_1,\ldots,h_m$.

The group $G$ acts on the space of all harmonic functions via $g\cdot f(x)=f(g^{-1}x)$. The image of a minimal harmonic function under this action is another minimal harmonic function, and so in particular for each $i=1,\ldots,m$ and each $g\in G$ we have some $\alpha_{g,i}\in\R$ and some $g\cdot i\in[m]$ such that $g\cdot h_i=\alpha_{g,i} h_{g\cdot i}$. As the notation $g\cdot i$ implicitly suggests, this defines an action of $G$ on the set $[m]$.

By the orbit-stabiliser theorem, for each $i$ the stabiliser $H_i$ of $i$ is of finite index in $G$; by Lemma \ref{lem:normal}, we may set $H$ to be a normal subgroup of $G$ that has finite index in $\bigcap_{i=1}^mH_i$, and hence in $G$. For every $g\in H$ we have $g\cdot h_i=\alpha_{g,i}h_i$, which is to say that $h_i(g^{-1}x)=\alpha_{g,i}h_i(x)$ for every $x\in G$ and every $i$. Taking $x=e$, and noting that $h_i(e)=1$, we see that $\alpha_{g,i}=h_i(g^{-1})$, and so this implies that $h_i(g^{-1}x)=h_i(g^{-1})h_i(x)$ for every $g\in H$ and every $x\in G$.

This implies that the restriction of $h_i$ to $H$ is a homomorphism into $\R^\times$, and moreover that $h_i(cx)=h_i(x)$ for every $c\in[H,H]$ and every $x\in G$. We conclude that each $h_i$ factors through $G/[H,H]$ (noting that $[H,H]$ is characteristic in $H$, and hence normal in $G$).

Let $p:G\to\R$ be a positive harmonic function. Since $p$ can be expressed in the form (\ref{eq:Martin}), $p$ must also factor through $G/[H,H]$, and so writing $\phi:G\to G/[H,H]$ we have $p=\hat p\circ\phi$, with $\hat p:G/[H,H]\to\R$ harmonic by Lemma \ref{lem:harm.pullback}. However, the abelian group $H/[H,H]$ is of finite index in $G/[H,H]$, and so Proposition \ref{prop:v.nilp.no.pos.harm} therefore implies that $\hat p$, and hence $p$, is constant.
\end{proof}
\section{Random walks on virtually cyclic groups}\label{sec:v.cyc}
In this section we consider an infinite group $G$ with a finite-index normal cyclic subgroup $\Z$ and a symmetric, finitely supported generating probability measure $\mu$. In a similar fashion to Section \ref{sec:v.ab}, we consider a finite set $T$ such that each $g\in G$ can be expressed uniquely as $g=\zeta(g)\tau(g)$ with $\zeta(g)\in\Z$ and $\tau(g)\in T$.

In general we continue to denote the identity of $G$ by $e$, the inverse of an element $g$ by $g^{-1}$, and the composition of two group elements $g,h$ by $gh$. However, when composing elements of $\Z$ with one another we often switch to additive notation to emphasise the integer structure. Thus, for example, we sometimes denote the identity by $0$, the inverse of $m$ by $-m$ and the composition of $m$ and $n$ by $m+n$, provided $m,n\in\Z$. This should not cause confusion since, whilst the notation for a given group element is not unique, neither is it ambiguous (in particular, we never multiply together two elements of $\Z$). For the avoidance of doubt, the notation $1$ always represents a generating element of the subgroup $\Z$, and never the identity element of $G$.

For each $n\in\N$ write $T^+_n=\min\{t\ge0:\zeta(X_t)\ge n\}$ and $T^-_n=\min\{t\ge0:\zeta(X_t)\le n\}$, noting that these quantities are almost surely finite. The purpose of this section is then to prove the following result.
\begin{lemma}\label{lem:rand.walk.cyc.linear}Let $m\in\Z$, and suppose that $g\in G$ with $m<\zeta(g)<m+R$. Let $M$ be as in Lemma \ref{lem:boundary.in.virt.cyc}. Then
\[
\Prob_g\left[\,T^+_{m+R}<T^-_m\,\right]=\frac{\zeta(g)-m}{R+M}+O\left(\frac{1}{R}\right)
\]
\end{lemma}
\begin{proof}
By Lemma \ref{lem:H1G} there exists a function $\varphi:T\to\R$ such that the function $f:G\to\R$ given by $f(nt)=n+\varphi(t)$ is harmonic on $G$. Let $t_{\min}\in T$ be the point at which $\varphi$ takes its minimum value, and $t_{\max}\in T$ the point at which $\varphi$ takes its maximum value, and define two further harmonic functions $f^+,f^-:G\to\R$ by
\[
f^+=\frac{1}{R+M}\Big(f-f((m-M)t_{\min})\Big);\qquad f^-=\frac{1}{R+M}\Big(f-f((m+R+M)t_{\max})\Big)+1.
\]
Note the following properties of $f^+,f^-$.
\begin{enumerate}
\renewcommand{\labelenumi}{(\roman{enumi})}
\item We have $f^-(nt)\le0\le f^+(nt)$ whenever $n\in[m-M,m]$.
\item We have $f^-(nt)\le1\le f^+(nt)$ whenever $n\in[m+R,m+R+M]$.
\end{enumerate}
Moreover, $f^+-f^-$ is constant and given by
\begin{equation}\label{eq:f+-f-}
f^+-f^-=\frac{\varphi(t_{\max})-\varphi(t_{\min})+M}{R+M},
\end{equation}
and we have
\begin{equation}\label{eq:gradient}
f^+((n+1)t)-f^+(nt)=f^-((n+1)t)-f^-(nt)=\frac{1}{R+M}
\end{equation}
for every $n\in\Z$ and every $t\in T$.

Now define $h:[m-M,m+R+M]T\to\R$ by setting
\[
h(nt)=\begin{cases}
0 &\text{when $n\in[m-M,m]$}\\
1 &\text{when $n\in[m+R,m+R+M]$},
\end{cases}
\]
and requiring that $h$ be harmonic elsewhere. Lemma \ref{cor:boundary} and the definition of $M$ imply that $h$ is well defined by these stipulations, and moreover that
\begin{equation}\label{eq:h.is.prob}
h(g)=\Prob_g\left[\,T^+_{m+R}<T^-_m\,\right].
\end{equation}
Now Corollary \ref{cor:dominating.harm.fn}, properties (i) and (ii) of $f^+,f^-$ and the definition of $h$ imply that $f^-\le h\le f^+$, and hence (i), (ii), (\ref{eq:f+-f-}) and (\ref{eq:gradient}) imply that
\[
h(g)=\frac{\zeta(g)-m}{R+M}+O\left(\frac{1}{R}\right)
\]
The desired result then follows from (\ref{eq:h.is.prob}).
\end{proof}
\section{Harmonic functions on groups with virtually cyclic quotients}
In this section we consider groups with virtually cyclic quotients. A well-known example of a group with a genuinely cyclic quotient is the \emph{lamplighter group}. If $L$ is the $\Z/2\Z$-vector space of finitely supported functions $\Z\to\Z/2\Z$, viewed as an additive group, then the \emph{lamplighter group} $G$ is the semidirect product $G=\Z\ltimes L$ defined by the action of $m\in\Z$ on $L$ given by $m\cdot f(x)=f(x-m)$. Explicitly, the group operation is defined by $(m,f)\cdot(m',f')=(m+m',f+m\cdot f')$.

I. Benjamini, G. Kozma and Yadin \cite{gady} give an explicit construction of a positive harmonic function on the lamplighter group.
\begin{prop}[Benjamini--Kozma--Yadin, unpublished]\label{prop:gady.ll}
Let $G$ be the lamplighter group, and let $\mu$ be a symmetric, finitely supported generating probability measure on $G$. Denote the random walk on the lamplighter group by $(M_0,F_0)$, $(M_1,F_1)$, $(M_2,F_2),\ldots$. Let $\tau_r=\min\{t\ge0:|M_t|\ge r\}$, and define $h_r:G\to\R$ by $h_r(g)=\Prob_g[\,F_{\tau_r}(n)=0\text{ for all }n<0\,]$. Then $rh_r$ converges pointwise to a positive harmonic function on $G$.
\end{prop}
In order to prove Theorem \ref{con:harm.fin.dim}, we need a slightly more general result. The purpose of this section is to show that the construction of Benjamini, Kozma and Yadin can be adapted fairly easily to obtain harmonic functions on a more general family of finitely generated groups with virtually cyclic quotients.
\begin{prop}\label{prop:gady'}
Let $G$ be a group with a symmetric, finitely supported generating probability measure $\mu$, and suppose that there is a homomorphism $\psi$ from $G$ onto an infinite virtually cyclic group such that $K=\ker\psi$ is not finitely generated. Then $(G,\mu)$ admits a positive harmonic function of at most linear growth that does not factor through $G/K$.
\end{prop}
\begin{remark}\label{rem:inf.dim.lin}
The function we construct in proving Proposition \ref{prop:gady'} is positive, and so Proposition \ref{prop:pos.harm.hom} implies that $G$ has an infinite-dimensional space spanned by positive harmonic functions, although we do not need this to prove Theorem \ref{con:harm.fin.dim}. It also implies that $H^1(G,\mu)$ is infinite dimensional, since if $\dim H^1(G,\mu)<\infty$ then every linearly growing harmonic function restricts to a homomorphism on some finite-index subgroup of $G$ \cite{mpty}.
\end{remark}
\begin{remark}
Meyerovitch and Yadin \cite{mey-yad} generalise Proposition \ref{prop:gady.ll} in another direction in proving their result that finite dimensionality of $H^1(G,\mu)$ for $G$ soluble implies that $G$ is virtually nilpotent.
\end{remark}
\bigskip

\noindent We start our proof of Proposition \ref{prop:gady'} by expressing $G$ in a particularly convenient form.
\begin{lemma}\label{lem:knt.decomp}
The group $G$ posseses an infinite cyclic subgroup $\Z$ such that $K\Z$ is normal in $G$, and a finite set $T$ containing the identity such that each $g\in G$ can be expressed uniquely as
\begin{equation}\label{eq:knt.decomp}
g=\kappa(g)\zeta(g)\tau(g)
\end{equation}
with $\kappa(g)\in K$, $\zeta(g)\in\Z$ and $\tau(g)\in T$. Moreover, $\psi$ is an isomorphism on $\Z$ and injective on $T$, and each $\overline g\in\psi(G)$ can be expressed uniquely as
\begin{equation}\label{eq:nt.decomp}
\overline g=\overline\zeta(\overline g)\overline\tau(\overline g)
\end{equation}
with $\overline\zeta(\overline g)\in\psi(\Z)$ and $\overline\tau(\overline g)\in\psi(T)$.
\end{lemma}
\begin{proof}The image $\psi(G)$ possesses an infinite cyclic subgroup $\langle z\rangle$ of finite index, and by Lemma \ref{lem:normal} we may assume that $\langle z\rangle$ is normal in $\psi(G)$. Let $\overline z\in\psi^{-1}(z)$. The element $\overline z$ is of infinite order, and we denote by $\Z$ the infinite cyclic subgroup that it generates. Note that $\psi$ is injective on $\Z$, and hence an isomorphism on $\Z$, as required.

Since $\psi(\Z)=\langle z\rangle$ is of finite index in $\psi(G)$, we may choose a finite set $\overline T$ containing $e$ such that each $\overline g\in\psi(G)$ can be expressed uniquely in the form (\ref{eq:nt.decomp}), with $\overline\tau(\overline g)\in\overline T$. For each $\overline t\in\overline T$ pick an arbitrary $t\in\psi^{-1}(\overline T)$, and define $T=\{t:\overline t\in\overline T\}$. It immediately follows that the element $\overline\tau(\overline g)$ in (\ref{eq:nt.decomp}) belongs to $\psi(T)$, and that $\psi$ is injective on $T$, as required. The injectivity of $\psi$ on $\Z$ additionally implies that each $g\in G$ can be expressed uniquely in the form (\ref{eq:knt.decomp}).

The fact that $K\Z$ is normal in $G$ follows immediately from the fact that $\psi(\Z)=\langle z\rangle$ is normal in $\psi(G)$.
\end{proof}
From now on in this section $\psi$ and $K$ are as in Proposition \ref{prop:gady'}, and $\Z$ and $T$ are fixed as in Lemma \ref{lem:knt.decomp}. Note that we have $\overline\tau(\psi(g))=\psi(\tau(g))$, and that if we abuse notation slightly and identify $\Z$ with its isomorphic image $\psi(\Z)$ we have $\overline\zeta(\psi(g))=\zeta(g)$.

As in Section \ref{sec:v.cyc}, when composing elements of $\Z$ or $\psi(\Z)$ with one another we often switch to additive notation to emphasise the integer structure.

Since $K$ is normal, the group $\Z$ acts on $K$ by conjugation. We may therefore define an automorphism $\varphi:K\to K$ by $\varphi(k)=1k1^{-1}$. More generally, this means that $\varphi^n(k)=nkn^{-1}$. As in Section \ref{sec:v.ab}, we denote $S:=\supp\mu$.

If $g=knt$ is a group element with $k\in K$, $n\in\Z$ and $t\in T$, then the elements adjacent to $g$ in the Cayley graph $(G,S)$ are the elements $gs=knts$ with $s\in S$. 
\begin{lemma}\label{lem:knts}
Let $k\in K$, $n\in\Z$, $t\in T$ and $s\in G$. Then
\[
\kappa(knts)=k\varphi^n(\kappa(ts));\qquad\zeta(knts)=n+\zeta(ts);\qquad\tau(knts)=\tau(ts).
\]
\end{lemma}
\begin{proof}
Expressing $ts$ in the form (\ref{eq:knt.decomp}), we have $knts=kn\kappa(ts)\zeta(ts)\tau(ts)$, and hence $knts=k\varphi^n(\kappa(ts))n\zeta(ts)\tau(ts)$, as claimed.
\end{proof}
For each set $A\subset\Z$ define a subgroup $U_A$ of $K$ by
\[
U_A=\langle\,\varphi^n(\kappa(ts)):s\in S,t\in T,n\in A\,\rangle,
\]
and for each $n\in\Z$ abbreviate by $U_n$ the subgroup $U_n=U_{[n,\infty)}$. Lemma \ref{lem:knts} implies that
\begin{equation}\label{eq:N=UZ}
K=U_\Z.
\end{equation}
\begin{lemma}\label{lem:decreasing.seq}
If $K$ is not finitely generated then, possibly after relabelling each $n\in\Z$ as $-n$, we have
\begin{equation}\label{eq:decreasing.seq}
\cdots\supsetneq U_{-2}\supsetneq U_{-1}\supsetneq U_0\supsetneq U_1\supsetneq U_2\supsetneq\cdots.
\end{equation}
\end{lemma}
\begin{proof}The containments of (\ref{eq:decreasing.seq}) are immediate by definition, so we just need to prove that they are strict. We start by showing that either $U_{\{0\}}\not\subset U_\N$ or $U_{\{0\}}\not\subset U_{-\N}$. Indeed, suppose that $U_{\{0\}}\subset U_\N$ and $U_{\{0\}}\subset U_{-\N}$, which, since $U_{\{0\}}$ is finitely generated, implies in particular that there is some $M\in\N$ such that
\begin{equation}\label{eq:U0.in.UN}
U_{\{0\}}\subset U_{[M]};
\end{equation}
\begin{equation}\label{eq:U0.in.U-N}
U_{\{0\}}\subset U_{-[M]}.
\end{equation}
Since $\varphi$ is an automorphism, (\ref{eq:U0.in.UN}) also implies that $U_{\{-1\}}\subset U_{\{0\}\cup[M-1]}$, and hence by (\ref{eq:U0.in.UN}) that $U_{\{-1\}}\subset U_{[M]}$. Repeating this argument, we conclude that $U_{\{-n\}}\subset U_{[M]}$ for every $n\in\N$. Similarly, (\ref{eq:U0.in.U-N}) implies that $U_{\{n\}}\subset U_{-[M]}$ for every $n\in\N$, and so in fact we have $U_\Z=U_{[-M,M]}$. By (\ref{eq:N=UZ}), this contradicts the assumption that $K$ is not finitely generated, and so either $U_{\{0\}}\not\subset U_\N$ or $U_{\{0\}}\not\subset U_{-\N}$, as claimed. Upon relabelling each $n\in\Z$ by $-n$ if necessary, we may assume the former, which implies in particular that $U_0\not\subset U_1$. Repeatedly using the fact that $\varphi$ is an automorphism then yields the lemma.
\end{proof}
We assume from now on that $\Z$ is labelled in such a way that (\ref{eq:decreasing.seq}) holds.

As usual, we denote by $X_0,X_1,X_2,\ldots$ the random walk on $G$ defined by $\mu$. In this section, we additionally denote by $\overline X_0,\overline X_1,\overline X_2,\ldots$ the random walk on $\psi(G)$ defined by $\psi(\mu)$. Note that the projected walk $(\psi(X_t))$ is isomorphic to the random walk $(\overline X_t)$.

For each $n\in\N$, write
\[
T^+_n=\min\{t\ge0:\zeta(X_t)\ge n\},\qquad\qquad T^-_n=\min\{t\ge0:\zeta(X_t)\le n\},
\]
\[
\overline T^+_n=\min\{t\ge0:\overline\zeta(\overline X_t)\ge n\},\qquad\qquad\overline T^-_n=\min\{t\ge0:\overline\zeta(\overline X_t)\le n\},
\]
noting that each of these quantities is almost surely finite. Note, incidentally, that if we identify $\overline X_t=\psi(X_t)$ then $T_n^\pm=\overline T_n^\pm$.

Define $B_R=\min\{t\ge0:\zeta(X_u)\ge0\text{ for all }u\in[t,T^+_R]\}$. More generally, for each $n<R$ set $B^n_R=\min\{t\ge0:\zeta(X_u)\ge n\text{ for all }u\in[t,T^+_R]\}$.
\begin{lemma}\label{lem:change.coset}
There exist some $l>\max\zeta(TS)$ and some $\alpha\in(0,1)$ such that if $R>l$, and if $g$ is such that $-l\le\zeta(g)\le0$, then either
\[
\Prob_g[\,\zeta(X_t)\ge-l\text{ for all }t\le T^+_R\,]=0
\]
or
\[
\Prob_g[\,\kappa(X_{B_R})\in U_0\,|\,\zeta(X_t)\ge-l\text{ for all }t\le T^+_R\,]\le\alpha
\]
\end{lemma}
\begin{proof}
Fix an element $u\in K\backslash U_0$, and for each $t\in T$ and each $j$ satisfying $0\le j<\max\zeta(TS)$ fix a path
\[
x^{j,t}_0=e,x^{j,t}_1,x^{j,t}_2,\ldots,x^{j,t}_{r_{j,t}}=t^{-1}\varphi^{-j}(u)t
\]
from $e$ to $t^{-1}\varphi^{-j}(u)t$ in the Cayley graph $(G,S)$, chosen so that
\begin{equation}\label{eq:c.fin.to.1}
\zeta(tx_i^{j,t})<-\max\zeta(TS)
\end{equation}
for at least one $i$.

Let $l=(1+\max_{j,t}r_{j,t})\max\zeta(TS)$. Write $\gamma=\min_{s\in S}\mu(s)$, and set $\beta=\gamma^{\max_{j,t}r_{j,t}}$. Note that for each $j,t$ there is a probability of at least $\beta$ that the random walk starting at $e$ has $x^{j,t}_0,x^{j,t}_1,\ldots,x^{j,t}_{r_{j,t}}$ as an initial segment.

Write $A$ for the set of (finite) paths $p$ from $g$ whose images $\zeta(p)$ in $\Z$ finish at $R$ or above, but stay in the range $[-l,R-1]$ until then. If $A=\varnothing$ then $\Prob_g[\,\zeta(X_t)\ge-l\text{ for all }t\le T^+_R\,]=0$ and the lemma holds, and so we may assume that $A\ne\varnothing$. For each $p\in A$, write $\overline k_p\overline m_p\overline t_p$ for the final position of $p$, with $\overline k_p\in K$, $\overline m_p\in\Z$ and $\overline t_p\in T$; thus $\overline m_p\ge R$, but all earlier positions of $\zeta(p)$ are below $R$. Also, let $\sigma_p$ be the largest final segment of $p$ whose image in $\Z$ lies entirely in the non-negative integers, and let $k_pm_pt_p$ be the first position of this final segment, with $k_p\in K$, $m_p\in\Z$ and $t_p\in T$. Note that
\begin{equation}\label{eq:mp<max}
0\le m_p<\max\zeta(TS).
\end{equation}
Lemma \ref{lem:knts} implies that $\{p\in A:k_p\in U_0\}=\{p\in A:\overline k_p\in U_0\}$, and so we may define $A_{\in}=\{p\in A:k_p\in U_0\}=\{p\in A:\overline k_p\in U_0\}$ and $A_{\notin}=A\backslash A_{\in}$. We claim that
\begin{equation}\label{eq:aim}
\Prob_g(A_{\notin})\gg\Prob_g(A_{\in}).
\end{equation}
This is sufficient to prove the lemma, since the conditional probability we are aiming to bound is equal to
\[
\frac{\Prob_g(A_{\in})}{\Prob_g(A_{\notin})+\Prob_g(A_{\in})}.
\]
We define a map $c$ from $A_{\in}$ to the set of finite paths starting at $g$ as follows. Given $p\in A_{\in}$, let $c(p)$ be the path that agrees with $p$ up until $k_pm_pt_p$, then has positions $k_pm_pt_px^{m_p,t_p}_1,\ldots,k_pm_pt_px^{m_p,t_p}_{r_{m_p,t_p}}$, and then continues with the same increments as the original path $p$ had after position $k_pm_pt_p$. This is well defined by (\ref{eq:mp<max}).

We claim that $c(p)\in A_{\notin}$ for every $p\in A_{\in}$. To see that $c(p)\in A$, note that
\begin{equation}\label{eq:first.perm.pos.c(p)}
k_pm_pt_px^{m_p,t_p}_{r_{m_p,t_p}}=k_pum_pt_p.
\end{equation}
This implies in particular that
\begin{equation}\label{eq:first.perm.pos.c(p)'}
\psi(k_pm_pt_px^{m_p,t_p}_{r_{m_p,t_p}})=m_p\psi(t_p)=\psi(k_pm_pt_p).
\end{equation}
By definition of $l$, at no point between $k_pm_pt_p$ and $k_pm_pt_px^{m_p,t_p}_{r_{m_p,t_p}}$ does $\zeta(p)$ drop below $-l$, and so it follows that $c(p)\in A$. To see, more specifically, that $c(p)\in A_{\notin}$, note that the definition of $k_pm_pt_p$ combines with (\ref{eq:first.perm.pos.c(p)'}) to imply that $\zeta(c(p))$ doesn't drop below zero after $k_pm_pt_px^{m_p,t_p}_{r_{m_p,t_p}}$. Lemma \ref{lem:knts} and (\ref{eq:first.perm.pos.c(p)}) therefore imply that $\overline k_{c(p)}$ is in the same left coset of $U_0$ as $k_pu$. In particular, since $k_p\in U_0$ and $u\notin U_0$ we have $\overline k_{c(p)}\notin U_0$, and so $c(p)\in A_{\notin}$, as claimed.

The fact that $c(A_{\in})\subset A_{\notin}$ of course implies that
\begin{equation}\label{eq:P(notin)>P(c(in))}
\Prob_g(A_{\notin})\ge\Prob_g(c(A_{\in})).
\end{equation}
We claim, moreover, that $c$ is $O(1)$-to-one. Write $a(p)$ for the segment that was added to $c$ to obtain $c(p)$, and note that one can, in principle at least, recover $p$ from $c(p)$ simply by deleting the segment $a(p)$. Note that (\ref{eq:c.fin.to.1}) and (\ref{eq:mp<max}) combine with Lemma \ref{lem:knts} and the fact (noted in the preceding paragraph) that $\zeta(c(p))$ doesn't drop below zero after $k_pm_pt_px^{m_p,t_p}_{r_{m_p,t_p}}$ to imply that $\zeta(p)$ drops below zero for the last time at some point during $a(p)$. This means that knowledge of $c(p)$ only is sufficient to identify, to within $\max_{j,t}r_{j,t}$ positions, where in $c(p)$ the segment $a(p)$ begins. Furthermore, the increments of $a(p)$ coincide with those of one of the finitely many paths $(x^{j,t}_i)$. There are therefore at most $O(1)$ possibilities for $a(p)$, given $c(p)$, and so $c$ is $O(1)$-to-one, as claimed.

This implies, in particular, that
\begin{equation}\label{eq:P(c(p))>>P(p)}
\Prob_g(c(A_{\in}))\gg\sum_{p\in A_{\in}}\Prob_g(c(p)).
\end{equation}
However, it follows from the definition of $\beta$ that for every $p\in A_{\in}$ we have $\Prob_g(c(p))\ge\beta\Prob_g(p)$. In combination with (\ref{eq:P(notin)>P(c(in))}) and (\ref{eq:P(c(p))>>P(p)}), this implies that
\[
\Prob_g(A_{\notin})\ge\Prob_g(c(A_{\in}))\gg\sum_{p\in A_{\in}}\Prob_g(c(p))\ge\beta\sum_{p\in A_{\in}}\Prob_g(p)=\beta\Prob_g(A_{\in}),
\]
and so (\ref{eq:aim}) holds as claimed and the lemma is proved.
\end{proof}
\begin{lemma}\label{lem:change.coset'}
Let $l$ and $\alpha$ be as given by Lemma \ref{lem:change.coset}. Let $n\le0$. Then if $R>l$, and if $g$ is such that $n-l\le\zeta(g)\le n$, then either
\[
\Prob_g[\,\zeta(X_t)\ge n-l\text{ for all }t\le T^+_R\,]=0
\]
or
\[
\Prob_g[\,\kappa(X_{B^n_R})\in U_n\,|\,\zeta(X_t)\ge n-l\text{ for all }t\le T^+_R\,]\le\alpha.
\]
\end{lemma}
\begin{proof}This follows immediately from applying Lemma \ref{lem:change.coset} with the weighted Cayley graph $(G,\mu)$ left-translated by $n$.
\end{proof}
\begin{lemma}\label{lem:exp.decay}
Let $l$ and $\alpha$ be as given by Lemma \ref{lem:change.coset}. Let $m\in\N$, and suppose that $k\le-ml$ and $R>l$. Then whenever $g\in G$ is such that $k\le\zeta(g)<k+l$ we have either
\[
\Prob_g[\,\zeta(X_t)\ge k\text{ for all }t\le T^+_R\,]=0
\]
or
\[
\Prob_g[\,\kappa(X_{T^+_R})\in U_0\,|\,\zeta(X_t)\ge k\text{ for all }t\le T^+_R\,]\le\alpha^m.
\]
\end{lemma}
\begin{proof}Everthing in this lemma is conditional on the event $\{\,\zeta(X_t)\ge k\text{ for all }t\le T^+_R\,\}$, so to make the notation less cumbersome we denote by $C_q$ the event
\[
C_q=\{\,\zeta(X_t)\ge q\text{ for all }t\le T^+_R\,\}.
\]
Applying Lemma \ref{lem:decreasing.seq}, we see that $\kappa(X_{T^+_R})\in U_0$ precisely when $\kappa(X_{B^n_R})\in U_n$ for each $n<0$. This implies in particular that
\[
\Big\{\,\kappa(X_{T^+_R})\in U_0\,\Big\}\subset\Big\{\,\kappa(X_{B^n_R})\in U_n\text{ for each }n=k+l,k+2l,\ldots,k+ml\,\Big\},
\]
and hence that it is sufficient to show that
\begin{equation}\label{eq:cond.prob.aim}
\Prob_g\left[\left.\,\kappa(X_{B^n_R})\in U_n\text{ for each }n=k+l,k+2l,\ldots,k+ml\,\right|\,C_k\,\right]\le\alpha^m
\end{equation}
whenever $\Prob_g[\,C_k\,]\ne0$. We show this by induction on $m$.

If $\Prob_g[\,C_k\,\wedge\,\{\,\kappa(X_{B^{k+l}_R})\in U_{k+l}\,\}\,]=0$ then either $\Prob_g[\,C_k\,]=0$ or the left-hand side of (\ref{eq:cond.prob.aim}) is $0$; in either case the lemma holds, so we may assume that $\Prob_g[\,C_k\,\wedge\,\{\,\kappa(X_{B^{k+l}_R})\in U_{k+l}\,\}\,]\ne0$. This implies that the left-hand side of (\ref{eq:cond.prob.aim}) is at most
\[
\begin{split}
\Prob_g\left[\left.\kappa(X_{B^{k+l}_R})\in U_{k+l}\,\right|\,C_k\,\right]\qquad\qquad\qquad\qquad\qquad\qquad\qquad\qquad\qquad\qquad\qquad\qquad\\
\times\,\Prob_g\left[\kappa(X_{B^n_R})\in U_n\text{ for }n=k+2l,\ldots,k+ml\,\left|\,C_k\,\wedge\,\left\{\,\kappa(X_{B^{k+l}_R})\in U_{k+l}\,\right\}\right.\,\right].
\end{split}
\]
However, it follows immediately from Lemma \ref{lem:change.coset'} that $\Prob_g[\kappa(X_{B^{k+l}_R})\in U_{k+l}\,|\,C_k\,]\le\alpha$, and so in fact the left-hand side of (\ref{eq:cond.prob.aim}) is at most
\[
\alpha\cdot\Prob_g\left[\kappa(X_{B^n_R})\in U_n\text{ for }n=k+2l,\ldots,k+ml\,\left|\,C_k\,\wedge\,\left\{\,\kappa(X_{B^{k+l}_R})\in U_{k+l}\right\}\right.\,\right].
\]
Conditioning on the position of the random walk on $G$ immediately after the projected walk on $\Z$ has left the set $[k,k+l-1]$ for the last time before reaching $R$, this is at most
\begin{equation}\label{eq:skopelos}
\alpha\cdot\frac{\sum_{y\,:\,\kappa(y)\in U_{k+l}}\Prob_y\left[\left.\kappa(X_{B^n_R})\in U_n\text{ for }n=k+2l,\ldots,k+ml\,\right|\,C_{k+l}\,\right]\cdot\Prob_g\left[\left.\,X_{B^{k+l}_R}=y\,\right|\,C_k\,\right]}{\Prob_g\left[\left.\,\kappa(X_{B^{k+l}_R})\in U_{k+l}\,\right|\,C_k\,\right]}.
\end{equation}
Note that if $\Prob_y[\,C_{k+l}\,]=0$ then $\Prob_g[\,X_{B^{k+l}_R}=y\,]=0$, so elements $y$ for which $\Prob_y[\kappa(X_{B^n_R})\in U_n\text{ for }n=k+2l,\ldots,k+ml\,|\,C_{k+l}\,]$ is not defined do not appear in the sum in the numerator of (\ref{eq:skopelos}), and so that sum is well defined. This means, moreover, that given $X_0=g$, for every possible value $y$ of $X_{B^{k+l}_R}$ the first factor of the summand of (\ref{eq:skopelos}) is at most $\alpha^{m-1}$ by induction, and so (\ref{eq:skopelos}), and hence the left-hand side of (\ref{eq:cond.prob.aim}), is at most $\alpha^m$, as required.
\end{proof}
Define $M_R=\min\{\zeta(X_t):t\le T_R^+\}$, so that $M_R$ is the minimum point hit by $\zeta(X_t)$ before it first exceeds $R$.
\begin{lemma}\label{lem:last.one}
Let $n\in\N$; let $l$ and $\alpha$ be as given by Lemma \ref{lem:change.coset}; let $m$ be such that $-(m+1)l<-n\le-ml$; and let $R>l$. Then either $\Prob_g[\,M_R=-n\,]=0$ or $\Prob_g[\,\kappa(X_{T^+_R})\in U_0\,|\,M_R=-n\,]\le\alpha^m$.
\end{lemma}
\begin{proof}
We may assume that $\Prob_g[\,M_R=-n\,]\ne0$, and hence in particular that $\Prob_g[\,C_{-n}\,]\ne0$, and so $\Prob_g[\,\kappa(X_{T^+_R})\in U_0\,|\,M_R=-n\,]$ is well defined and equal to
\[
\begin{split}
\sum_{y\in G}\Prob_g\left[\left.\,\left\{\,T_{\{b\in G:\zeta(b)=-n\}}<T^+_R\,\,\text{and}\,\, X_{T_{\{b\in G:\zeta(b)=-n\}}}=y\,\right\}\,\right|\,C_{-n}\,\right]\qquad\qquad\\
\qquad\qquad\qquad\qquad\qquad\qquad\qquad\qquad\qquad\qquad\times\,\Prob_y[\,\kappa(X_{T^+_R})\in U_0\,|\,C_{-n}\,],
\end{split}
\]
which is at most
\begin{equation}\label{eq:gecko}
\begin{split}
\sum_{y\in G\,:\,\zeta(y)=-n}\Prob_g\left[\left.\,X_{T_{\{b\in G:\zeta(b)=-n\}}}=y\,\right|\,\left\{\,T_{\{b\in G:\zeta(b)=-n\}}<T^+_R\,\,\text{and}\,\, C_{-n}\,\right\}\,\right]\qquad\qquad\\
\qquad\qquad\qquad\qquad\qquad\qquad\qquad\qquad\qquad\times\,\Prob_y[\,\kappa(X_{T^+_R})\in U_0\,|\,C_{-n}\,].
\end{split}
\end{equation}
If $\Prob_y[\,C_{-n}\,]=0$ then
\[
\Prob_g\left[\left.\,X_{T_{\{b\in G:\zeta(b)=-n\}}}=y\,\right|\,\left\{\,T_{\{b\in G:\zeta(b)=-n\}}<T^+_R\,\,\text{and}\,\, C_{-n}\,\right\}\,\right]=0,
\]
and so elements $y$ for which
\begin{equation}\label{eq:gecko.2}
\Prob_y[\,\kappa(X_{T^+_R})\in U_0\,|\,C_{-n}\,]
\end{equation}
is not defined do not appear in the sum (\ref{eq:gecko}) and that sum is well defined. The sum (\ref{eq:gecko}) is, moreover, the expectation of the quantity (\ref{eq:gecko.2}) with respect to some probability measure on the set $\{y\in G\,:\,\zeta(y)=-n\}$. However, for each $y$ in that set for which the quantity (\ref{eq:gecko.2}) is defined, the quantity (\ref{eq:gecko.2}) is at most $\alpha^m$ by Lemma \ref{lem:exp.decay}, and so (\ref{eq:gecko}) is at most $\alpha^m$ and the lemma is proved.
\end{proof}
Define a real-valued function $h_R$ on the subset $K[-R,R]T$ of $G$ by $h_R(g)=\Prob_g[\,T^+_R<T^-_{-R}\,\,\,\text{and}\,\,\,\kappa(X_{T^+_R})\in U_0\,]$.
\begin{lemma}\label{lem:gady.function}
The function $h_R$ satisfies the following properties.
\begin{enumerate}
\renewcommand{\labelenumi}{(\roman{enumi})}
\item The function $h_R$ is positive and harmonic on the interior of $K[-R,R]T$.
\item For every $g\in(K[-R,R]T)^\circ$ we have $h_R(g)\ll|\zeta(g)|/R$.
\item If $\zeta(g)\ge0$ and $\kappa(g)\notin U_0$ then $h_R(g)\ll1/R$.
\item If $\zeta(g)\ge0$ and $\kappa(g)\in U_0$ then $h_R(g)\gg\zeta(g)/R$.
\end{enumerate}
\end{lemma}
\begin{proof}The positivity and harmonicity of $h_R$ are clear, so we prove properties (ii), (iii) and (iv). We may rewrite $h_R(g)$ by conditioning on $M_R$ as follows:
\begin{equation}\label{eq:conditional}
\begin{split}
h_R(g)=\left(\sum_{n=0}^R\Prob_g[\,M_R=-n\,]\cdot\Prob_g[\,\kappa(X_{T^+_R})\in U_0\,|\,M_R=-n\,]\right)\qquad\\
+\,\Prob_g[\,M_R>0\,]\cdot\Prob_g[\,\kappa(X_{T^+_R})\in U_0\,|\,M_R>0\,].
\end{split}
\end{equation}
Let us examine these probabilities in turn, starting with $\Prob_g[\,M_R=-n\,]$. This corresponds to the event that $\zeta(X_t)$ hits $-n$ before reaching or exceeding $R$, but then reaches or exceeds $R$ before dropping below $-n$. In particular,
\begin{equation}\label{eq:first.prob.decomp}
\Prob_g[\,M_R=-n\,]\le\Prob_{\psi(g)}\left[\,\overline T^-_{-n}<\overline T^+_R\,\right]\cdot\max_{t\in T}\Prob_{(-n)\psi(t)}\left[\,\overline T^+_R<\overline T^-_{-(n+1)}\,\right]
\end{equation}
Applying Lemma \ref{lem:rand.walk.cyc.linear}, for each $n\ge0$ we have
\begin{equation}\label{eq:first.prob.1}
\Prob_{\psi(g)}\left[\,\overline T^-_{-n}<\overline T^+_R\,\right]=\frac{R-\zeta(g)}{R+n+O(1)}+O\left(\frac{1}{R}\right)
\end{equation}
and
\begin{equation}\label{eq:first.prob.2}
\max_{t\in T}\Prob_{(-n)\psi(t)}\left[\,\overline T^+_R<\overline T^-_{-(n+1)}\,\right]=\frac{1}{R+n+O(1)}+O\left(\frac{1}{R}\right).
\end{equation}
If $\zeta(g)\le0$ then of course $\Prob_g[\,M_R>0\,]=0$; another application of Lemma \ref{lem:rand.walk.cyc.linear} implies that more generally we have
\begin{equation}\label{eq:second.prob}
\Prob_g[\,M_R>0\,]=
\begin{cases}
\frac{\zeta(g)}{R+O(1)}+O\left(\frac{1}{R}\right) &\text{if $\zeta(g)>0$};\\
0 &\text{if $\zeta(g)\le0$}.
\end{cases}
\end{equation}
We now consider $\Prob_g[\,\kappa(X_{T^+_R})\in U_0\,|\,M_R=-n\,]$ when $n\in\N$ and $\Prob_g[\,M_R=-n\,]\ne0$. Let $l$ and $\alpha$ be as given by Lemma \ref{lem:change.coset}, noting in particular that $\alpha<1$, and let $m$ be such that $-(m+1)l<-n\le-ml$. Lemma \ref{lem:last.one} then implies that
\begin{equation}\label{eq:cond.prob.1}
\Prob_g[\,\kappa(X_{T^+_R})\in U_0\,|\,M_R=-n\,]\le\alpha^m.
\end{equation}
Finally, the condition that $M_R>0$ implies that $\zeta(X_t)$ does not drop below zero until after time $T^+_R$, which by Lemma \ref{lem:knts} means that $\kappa(X_t)$ is in the same left coset of $U_0$ as $\kappa(g)$ for every $t\le T^+_R$. We therefore have
\begin{equation}\label{eq:cond.prob.2}
\Prob_g[\,\kappa(X_{T^+_R})\in U_0\,|\,M_R>0\,]=
\begin{cases}
1 &\text{if $\kappa(g)\in U_0$}\\
0 &\text{otherwise}.
\end{cases}
\end{equation}
Properties (ii), (iii) and (iv) then follow from (\ref{eq:conditional}), (\ref{eq:first.prob.decomp}), (\ref{eq:first.prob.1}), (\ref{eq:first.prob.2}), (\ref{eq:second.prob}), (\ref{eq:cond.prob.1}) and (\ref{eq:cond.prob.2}) and the fact that $\alpha<1$.
\end{proof}
\begin{proof}[Proof of Proposition \ref{prop:gady'}]
Property (ii) of Lemma \ref{lem:gady.function} implies that $R\cdot h_R(g)=O(|\zeta(g)|)$, so for each $g$ there is a convergent subsequence of $R\cdot h_R(g)$ as $R\to\infty$. Since $G$ is countable, a simple diagonal argument therefore gives a subsequence of $R\cdot h_R$ that converges pointwise to a function $h:G\to\R$, which grows at most linearly in $|g|$ by the bound from Lemma \ref{lem:gady.function} (ii). The limit function $h$ is harmonic by property (i) of Lemma \ref{lem:gady.function}, and does not factor through $G/K$ by properties (iii) and (iv).
\end{proof}
\section{Groups with finite-dimensional spaces of harmonic functions}\label{sec:proof}
In this section we prove Theorem \ref{con:harm.fin.dim}. The group $G$ in Theorem \ref{con:harm.fin.dim} acts on the space $H$ of harmonic functions on $(G,\mu)$ via the linear transformations $g\cdot f(x)=f(g^{-1}x)$. This action defines a homomorphism $G\to GL(H)$, which we denote by $\psi:G\to GL(H)$ throughout this section.
\begin{lemma}\label{lem:harm.factors}
A function $h:G\to\R$ is harmonic with respect to $\mu$ if, and only if, there is some function $\overline{h}:\psi(G)\to\R$, harmonic with respect to $\psi(\mu)$, such that $h=\overline{h}\circ\psi$. Moreover, $h\in H^k(G,\mu)$ if and only if $\overline h\in H^k(\psi(G),\psi(\mu))$.
\end{lemma}
\begin{proof}If $h:G\to\R$ is harmonic and $k\in\ker\psi$ then $h(kg)=h(g)$ for every $g$, so there exists $\overline{h}:\psi(G)\to\R$ such that $h=\overline{h}\circ\psi$. It is easy to see that $\overline h$ exhibits polynomial growth of degree at most $k$ if and only if $h$ does, so the desired result then follows from Lemma \ref{lem:harm.pullback}.
\end{proof}
\begin{proof}[Proof of Theorem \ref{con:harm.fin.dim}]
Suppose first that $G$ has a finite-index infinite nilpotent subgroup $N$ of \emph{rank} $d\in\N$ (the rank is defined in \cite{mpty}, for example, and is equal to $1$ if and only if $N$ is virtually cyclic). It follows from \cite{mpty} that $\dim H^k(G,\mu)\gg_dk^{d-1}$, which implies one direction of the theorem, the other direction being Proposition \ref{prop:direct}.

Now suppose that $G$ is not virtually nilpotent. If the space of harmonic functions is finite dimensional then $\psi$ may be viewed as a homomorphism $\psi:G\to GL_n(\R)$. By Lemma \ref{lem:harm.factors}, the space of harmonic functions on $(\psi(G),\psi(\mu))$ is finite dimensional, and so, by Corollary \ref{cor:linear}, $\psi(G)$ is virtually nilpotent. It is therefore virtually cyclic by the virtually nilpotent case of the theorem.

If $\psi(G)$ is finite then, by the maximum principle (Lemma \ref{lem:max.princ}) and Lemma \ref{lem:harm.factors}, every harmonic function on $(G,\mu)$ is constant, contradicting Corollary \ref{cor:harm.exists}. Thus $\psi(G)$ is infinite. If $\ker\psi$ is not finitely generated, Proposition \ref{prop:gady'} therefore gives a harmonic function on $(G,\mu)$ that does not factor through $G/\ker\psi$. On the other hand, since Proposition \ref{prop:varopoulos} implies that the random walk on $(G,\mu)$ is transient, if $\ker\psi$ is finitely generated and infinite then Proposition \ref{prop:harm.exists} gives a harmonic function that is not constant on $\ker\psi$. In either case this contradicts Lemma \ref{lem:harm.factors}, and so $\ker\psi$ must in fact be finite. Since $\psi(G)$ is virtually cyclic, it follows that $G$ is itself virtually cyclic, and the theorem holds.
\end{proof}
\begin{proof}[Proof of Corollary \ref{cor:osin}]
If $G$ is not amenable then the space of bounded harmonic functions is infinite dimensional \cite{kai-ver}. If $G$ is virtually nilpotent then the result follows from the same argument as for Theorem \ref{con:harm.fin.dim}. Osin \cite[Proposition 3.1]{osin} has shown that if $G$ is elementary amenable and not virtually nilpotent then it has a normal subgroup $H$ such that $G/H$ is virtually polycyclic, and virtually nilpotent only if $H$ is not finitely generated. If $G/H$ is not virtually nilpotent then the quotient has an infinite dimensional space of harmonic functions of linear growth \cite{mey-yad}, and so the corollary follows from Lemma \ref{lem:harm.factors}. If $G/H$ is virtually nilpotent and not virtually cyclic then the corollary follows from Lemma \ref{lem:harm.factors} and the virtually nilpotent case. Finally, if $G/H$ is virtually cyclic then the corollary follows from Proposition \ref{prop:gady'} (see Remark \ref{rem:inf.dim.lin}).
\end{proof}
\begin{remarks}
Meyerovitch and Yadin's result \cite{mey-yad} could also be used in place of Corollary \ref{cor:linear} in the proof of Theorem \ref{con:harm.fin.dim}.

It is conjectured that if $G$ is \emph{any} non-virtually nilpotent group with a symmetric, finitely supported generating probability measure $\mu$ then $\dim H^1(G,\mu)=\infty$ \cite{mey-yad}. A verification of this conjecture would immediately reduce both Theorem \ref{con:harm.fin.dim} and Conjecture \ref{con:osin} to the virtually nilpotent case, which, in each case, follows from the results of \cite{mpty} as described above.
\end{remarks}
\appendix
\section{Further applications of our Garden of Eden theorem}\label{ap:GofE}
In this appendix we use Theorem \ref{thm:GofE.neighbourly} to recover Theorem \ref{thm:GofE} and to reformulate a conjecture of I. Kaplansky.
\subsection*{The Ceccherini-Silberstein--Coornaert Garden of Eden theorem}
Theorem \ref{thm:GofE} follows immediately from Theorem \ref{thm:GofE.neighbourly} and the following result.
\begin{prop}\label{prop:pre-inj.iff.transp.pre-inj}Let $V$ be a finite-dimensional vector space and let $\tau:V^G\to V^G$ be a linear cellular automaton with memory set $M$ over an amenable group $G$. Then $\tau$ is pre-injective if and only if $\tau'$ is pre-injective.
\end{prop}
\begin{remark}\label{rem:countable}As we noted at the start of the proof of Theorem \ref{thm:GofE.neighbourly}, a locally specifiable map on a locally finite graph is pre-injective if and only if it is pre-injective on every connected component; in proving Proposition \ref{prop:pre-inj.iff.transp.pre-inj} we may therefore assume that $G$ is generated by $M$, and hence that $G$ is countable.
\end{remark}
From now on in this appendix, $G$ is a fixed countable amenable group and $V$ is a fixed finite-dimensional vector space.

In proving Theorem \ref{thm:GofE}, Ceccherini-Silberstein and Coornaert make use of the notion of \emph{mean dimension}, the use of which in connection to Theorem \ref{thm:GofE} appears to have been first suggested by Gromov \cite[\S8.J]{gromov'}.

Let $X$ be a subspace of $V^G$. Given a subset $\Omega$ of $G$ and an element $f$ of $V^G$, denote by $f_\Omega$ the function that agrees with $f$ on the subset $\Omega$ and takes the value $0$ elsewhere, and denote by $X_\Omega$ the subspace of $V^G$ defined by $X_\Omega=\{f_\Omega:f\in X\}$. Since $G$ is countable and amenable, it admits a \emph{F\o lner sequence}, which is to say a sequence $(\Omega_n)_{n\in\N}$ of subsets of $G$ with the property that for every $g\in G$ we have $\frac{|\Omega_n\,\triangle\,\Omega_n g|}{|\Omega_n|}\to0$ as $n\to\infty$ \cite{folner}. This implies in particular that
\begin{equation}\label{eq:Folner}
\frac{|\partial^+\Omega_n|}{|\Omega_n|}\to0.
\end{equation}
The \emph{mean dimension} of $X$ with respect to $(\Omega_n)_{n\in\N}$ is then denoted $\mdim X$, and defined by $\mdim X=\liminf_{n\to\infty}\frac{\dim X_{\Omega_n}}{|\Omega_n|}$.

For the remainder of this appendix, $(\Omega_n)_{n\in\N}$ is a fixed F\o lner sequence in $G$, and the mean dimension of a subspace $X$ of $V^G$ is always computed with respect to $(\Omega_n)_{n\in\N}$. We define the neighbourhood $\Omega^+$ of a subset $\Omega\subset G$ to be its neighbourhood in the Cayley graph $(G,M)$.

Ceccherini-Silberstein and Coornaert \cite{GofE} originally obtained Theorem \ref{thm:GofE} in the case of a countable amenable group as an immediate consequence of the following more precise statement.
\begin{prop}[Ceccherini-Silberstein--Coornaert {\cite[Theorem 4.10]{GofE}}]\label{prop:mdim}
Let $\tau:V^G\to V^G$ be a linear cellular automaton. Then the following statements are equivalent:
\begin{enumerate}
\renewcommand{\labelenumi}{(\arabic{enumi})}
\item $\tau$ is surjective;
\item $\tau$ is pre-injective;
\item $\mdim\tau(V^G)=\dim V$.
\end{enumerate}
\end{prop}
The key observation that allows us to prove Proposition \ref{prop:pre-inj.iff.transp.pre-inj} is that the mean dimension of $\tau$ is equal to that of its transpose $\tau'$.
\begin{prop}\label{prop:new.mdim}
Let $\tau:V^G\to V^G$ be a locally specifiable linear map, with local specifiability defined in terms of the Cayley graph $(G,M)$. Then $\mdim\tau'(V^G)=\mdim\tau(V^G)$.
\end{prop}
In proving Proposition \ref{prop:new.mdim}, we make use of the following straightforward lemma.
\begin{lemma}\label{lem:mdim}
Let $X$ be a subspace of $V^G$. Then $\dim X_{\Omega_n^+}=\dim X_{\Omega_n}+o(|\Omega_n|)$.
\end{lemma}
\begin{proof}
We have $X_{\Omega_n^+}\subset X_{\Omega_n}\oplus V^G_{\partial^+\Omega_n}$, and so $\dim X_{\Omega_n^+}\le\dim X_{\Omega_n}+\dim V^G_{\partial^+\Omega_n}=\dim X_{\Omega_n}+|\partial^+\Omega_n|\dim V$ and the desired result follows from (\ref{eq:Folner}).
\end{proof}
Given a locally specifiable linear map $\tau:V^G\to V^G$ and finite subsets $A,B\subset G$, we denote by $\tau_B^A$ the $|B|\times|A|$ matrix formed by taking the rows of $\tau$ corresponding to elements of $B$ and the columns of $\tau$ corresponding to elements of $A$.
\begin{proof}[Proof of Proposition \ref{prop:new.mdim}]
Note that $\dim\tau(V^G)_{\Omega_n}=\dim\tau(V^G_{\Omega_n^+})_{\Omega_n}$ and
\[
\dim\tau'(V^G)_{\Omega_n}=\dim\tau'(V^G_{\Omega_n^+})_{\Omega_n},
\]
which, by Lemma \ref{lem:mdim}, implies that
\begin{equation}\label{eq:dims.close}
\dim\tau(V^G)_{\Omega_n}-\dim\tau'(V^G)_{\Omega_n}
=\dim\tau(V^G_{\Omega_n^+})_{\Omega_n^+}-\dim\tau'(V^G_{\Omega_n^+})_{\Omega_n^+}+o(|\Omega_n|).
\end{equation}
However, $\tau(V^G_{\Omega_n^+})_{\Omega_n^+}$ is isomorphic to the image of $\tau_{\Omega_n^+}^{\Omega_n^+}$, and $\tau'(V^G_{\Omega_n^+})_{\Omega_n^+}$ is isomorphic to the image of $(\tau')_{\Omega_n^+}^{\Omega_n^+}$. Since $\tau_{\Omega_n^+}^{\Omega_n^+}$ and $(\tau')_{\Omega_n^+}^{\Omega_n^+}$ are finite and transposes of one another, this implies that $\dim\tau(V^G_{\Omega_n^+})_{\Omega_n^+}=\dim\tau'(V^G_{\Omega_n^+})_{\Omega_n^+}$, and so (\ref{eq:dims.close}) implies that $\dim\tau(V^G)_{\Omega_n}-\dim\tau'(V^G)_{\Omega_n}=o(|\Omega_n|)$. The desired result then follows immediately from the definition of mean dimension.
\end{proof}
\begin{proof}[Proof of Proposition \ref{prop:pre-inj.iff.transp.pre-inj}]
By Remark \ref{rem:countable} we may assume that $G$ is generated by $M$ and, in particular, that $G$ is countable. Proposition \ref{prop:pre-inj.iff.transp.pre-inj} then follows directly from Lemma \ref{prop:new.mdim} and the equivalence (2) $\Leftrightarrow$ (3) of Proposition \ref{prop:mdim}. The equivalence (2) $\Leftrightarrow$ (3) of Proposition \ref{prop:mdim} follows from \cite[Lemmas 4.8 \& 4.9]{GofE}.
\end{proof}
\begin{remark}\label{rem:GofE.proof}
It would be stretching reality somewhat to claim that this represented a new proof of Theorem \ref{thm:GofE}, since there is considerable overlap between our proof of Proposition \ref{prop:pre-inj.iff.transp.pre-inj} and  Ceccherini-Silberstein and Coornaert's original proof of Theorem \ref{thm:GofE}. However, arranging the proof in this way probably shortens the proof slightly, and perhaps makes clearer the role of amenability; note, in particular, that it is only in using the mean-dimension to convert a statement about $\tau'$ to a statement about $\tau$ that we use the amenability of $G$.
\end{remark}
\subsection*{Kaplansky's stable-finiteness conjecture}
A group $G$ is called \emph{linear surjunctive} if every injective linear cellular automaton is surjective. Since injectivity is stronger than pre-injectivity, Theorem \ref{thm:GofE} immediately implies that an amenable group is linear surjunctive. Ceccherini-Silberstein and Coornaert \cite[Theorem 8.14.4]{LCA} have shown, more generally, that every sofic group is linear surjunctive. They also note that linear surjunctivity of a group $G$ is related to a certain condition on group algebras, called \emph{stable finiteness}, as follows. We refer the reader to \cite[\S8]{LCA} for a definition of stable finiteness, and for further background.
\begin{prop}[{\cite[Corollary 8.15.6]{LCA}}]
Let $G$ be a group and let $\K$ be a field. Then the following conditions are equivalent.
\begin{enumerate}
\renewcommand{\labelenumi}{(\arabic{enumi})}
\item For every finite-dimensional vector space $V$ over $\K$, every injective linear cellular automaton $\tau:V^G\to V^G$ is surjective.
\item The group algebra $\K[G]$ is stably finite.
\end{enumerate}
\end{prop}
In particular, if $G$ is a sofic group and $\K$ is a field then the group algebra $\K[G]$ is stably finite. It is natural to ask whether this holds for more general groups; Ceccherini-Silberstein and Coornaert \cite[p. 418, (OP-15)]{LCA} attribute this question to Kaplansky.
\begin{question}[Kaplansky]\label{qu:kap}
Do either, and hence both, of the following equivalent statements hold?
\begin{enumerate}
\renewcommand{\labelenumi}{(\arabic{enumi})}
\item For any group $G$ and field $\K$ the group algebra $K[G]$ is stably finite.
\item Every group is linear surjunctive.
\end{enumerate}
\end{question}
By Theorem \ref{thm:GofE.neighbourly} this question can be reformulated as follows.
\begin{corollary}
Statements (1) and (2) of Question \ref{qu:kap} are equivalent to the following statement.
\begin{enumerate}
\renewcommand{\labelenumi}{(\arabic{enumi})}
\setcounter{enumi}{2}
\item If $\tau$ is an injective linear cellular automaton over an arbitrary group then its transpose $\tau'$ is pre-injective.
\end{enumerate}
\end{corollary}

\end{document}